\documentclass[draft]{birkjour}
\usepackage[noadjust]{cite}
\usepackage{xcolor}
\usepackage{mathrsfs}

\usepackage{amsmath}
\usepackage{amssymb}
\usepackage{amsthm}
\usepackage{extarrows}
\RequirePackage[all]{xy}

\newtheorem{theorem}{Theorem}[section]
\newtheorem{corollary}[theorem]{Corollary}
\newtheorem{lemma}[theorem]{Lemma}
\newtheorem{proposition}[theorem]{Proposition}
\theoremstyle{definition}
\newtheorem{definition}[theorem]{Definition}
\theoremstyle{remark}
\newtheorem*{remark}{Remark}

\numberwithin{equation}{section}

\newcommand{\BibTeX}{B\kern-0.1emi\kern-0.017emb\kern-0.15em\TeX}
\newcommand{\XYpic}{$\mathrm{X\kern-0.3em\raisebox{-0.18em}{Y}}$-$\mathrm{pic}\,$}

\newcommand{\cl}{C \kern -0.1em \ell}  

\newcommand{\be}{\begin{eqnarray*}}
	\newcommand{\ee}{\end{eqnarray*}}

\allowdisplaybreaks

\newcommand{\ed}{\end{document}}
\begin{document}

%
%
%
%
%
%
%
%
%

\title[Integral formulas and Hodge decomposition]
{Integral formulas and Hodge decomposition in the theory of generalized partial-slice mo-nogenic functions}
\author[Manjie Hu]{Manjie Hu}
\address{%
	School of Mathematical Sciences,\\ Anhui University, Hefei, P.R. China}
\email{A23201033@stu.ahu.edu.cn}

\author[Chao Ding]{Chao Ding}
%
\address{%
	Center for Pure Mathematics, \\School of Mathematical Sciences,\\ Anhui University, Hefei, P.R. China}
\email{cding@ahu.edu.cn}

%
\date{\today}
\begin{abstract}
	This paper explores generalized slice monogenic functions by introducing their operator symbols, representation formula, and integral formula. The study extends the Teodorescu transform to a broader class of theorems and inferences, providing new analytical tools for function theory in this setting. Additionally, the Hodge decomposition is established, providing a foundation for further research.
	
\end{abstract}
\label{page:firstblob}
\maketitle
\section{Introduction}
The study of function theories over Clifford algebras has been a rich and evolving field with significant implications in both pure and applied mathematics, including mathematical physics, engineering, and differential equations \cite{Di2,Br1}. Over the last century, Clifford analysis has become an important tool for solving differential equations and boundary value problems. This includes monogenic functions, which are solutions to the Weyl or Dirac systems, defined in $\mathbb{R}^{n+1} $ or $ \mathbb{R}^n $ with values in the real Clifford algebra $ \mathbb{R}_n $. These functions, linked to a generalized Cauchy–Riemann operator, extend holomorphic functions and have been studied since the late nineteenth century. For more details, see \cite{Co8,Ji}.
\par
This paper is motivated by recent advances in the theory of slice monogenic functions, particularly the integral formula and representation theorem established in \cite{Co2,Co1,Fu,Ge3}.  Additionally, global differential operators introduced in \cite{Gh2,Co3} have contributed to the study of slice regularity, while conformal invariance properties discussed in \cite{Ry} provide insight into the behavior of these functions under M\"obius transformations.
\par
The concept of slice functions, introduced by Gentili and Struppa \cite{Ge1,Ge2} in the early 2000s, differs from classical Fueter regularity in quaternionic analysis. Defined through power series expansions and integral representations, this approach extends holomorphicity to quaternionic, octonionic, and Clifford algebras. This generalization has led to the concept of slice monogenic functions \cite{Co2,Co5}, which share important properties of classical holomorphic functions, such as Cauchy integral formula and power series representations, while adapting to the non-commutativity and non-associativity of higher-dimensional algebras.
\par
Building on this foundation, Colombo, Sabadini, and Struppa formulated an alternative perspective on the theory of slice monogenic functions, enabling structural decompositions that facilitate representation theorem, integral formula, and functional calculus techniques. In \cite{Co6}, Colombo et al. also introduced a differential operator $G$ with non-constant coefficients, whose null solutions correspond to slice regular functions under certain conditions. This provides a direct way to define slice regular functions, which plays the same role as the Cauchy-Riemann operator does in complex analysis. Later, \cite{Xu2} introduced two such operators, $G_{\boldsymbol{x}}$ and $\overline{\vartheta }$, to study generalized partial-slice monogenic functions on partially symmetric domains.
\par
The integral over the boundary in the complex Cauchy–Pompeiu representation defines a weakly singular integral operator $T_{\Omega}$, which was studied by Vekua \cite{Ve}. In \cite{Zh1,Zh2}, the authors explained the Teodorescu transform as a solution to non-homogeneous Dirac equations in Clifford analysis. In \cite{Gu}, the authors introduces the Teodorescu transform systematically, and they particularly point out that $ T_{\Omega} $ acts like an left inverse operator for $\overline{\partial }$ when acting on functions with compact support. Further, this operator helps solve boundary value problems, create integral representations of monogenic functions, and extend function theories to generalized spaces like slice monogenic and polymonogenic functions. The purpose of this article is to investigate the Teodorescu transform in the theory of slice monogenic functions. Some mapping properties of the Teodorescu transform are also studied, and Hodge decomposition is given as an application.
\par
This paper is organized as follows. Section 2 reviews the fundamentals of Clifford algebras and introduces monogenic and slice monogenic functions. Section 3 introduces operator notations and important theorems on generalized slice monogenic functions. In Section 4, we use the integral formula of the generalized partial-slice monogenic function to estimate the norm of the Teodorescu transform. Section 5 presents the Hodge decomposition, which plays a important role in our research.

\section{Preliminaries}
This section provides a refined overview of Clifford algebras and the theories of monogenic and slice monogenic functions. For comprehensive details, refer to \cite{Br2,Co4,De,Ge1}.
\subsection{Clifford algebras}
Let $\{e_1, e_2, \cdots, e_n\}$ be an orthonormal basis of the $n$-dimensional real Euclidean space, denoted by $\mathbb{R}^n$. The associated Clifford algebra $\mathbb{R}_n$ is generated by this basis with the relationship
\begin{align*}
	e_i e_j + e_j e_i = -2 \delta_{ij}, \quad 1 \leq i, j \leq n,
\end{align*}
where $\delta_{ij}$ represents the Kronecker symbol. Any element $a \in \mathbb{R}_n$ is then expressible as
\begin{align*}
	a = \sum_A a_A e_A, \quad a_A \in \mathbb{R},
\end{align*}
where $e_A = e_{j_1} e_{j_2} \cdots e_{j_r}$ for index sets $A = \{j_1, j_2, \cdots, j_r\} \subseteq \{1, 2, \cdots, n\}$, with $1 \leqslant j_1 < j_2 < \cdots < j_r \leqslant n$, and $e_{\emptyset} = e_0 = 1$. Therefore, one can see that the Clifford algebra $\mathbb{R}_n$, considered as a real vector space, has dimension $2^n$.
\par
For each integer $k = 0, 1, \cdots, n$, the real linear subspace of $\mathbb{R}_n$, denoted $\mathbb{R}_n^k$, consists of $k$-vectors, which are generated by $\binom{n}{k}$ elements of the form
\begin{align*}
	e_A = e_{i_1} e_{i_2} \cdots e_{i_k}, \quad 1 \leqslant i_1 < i_2 < \cdots < i_k \leqslant n.
\end{align*}
The scalar part of $a$ is denoted by $[a]_0 = a_{\emptyset}$. The paravectors, a important subset of Clifford numbers $\mathbb{R}_n$, are formed by elements in $\mathbb{R}_n^0 \oplus \mathbb{R}_n^1$. This subset is mapped to $\mathbb{R}^{n+1}$ with the following mapping
\begin{align*}
	(x_0, x_1, \cdots, x_n) \mapsto x = x_0 + \underline{x} =x_0+ \sum_{i=1}^n e_i x_i.
\end{align*}

Next, we define two important involutions in the Clifford algebra $\mathbb{R}_n$ as follows. The first involution is called the \textbf{Clifford conjugation}. For any element $a = \sum_A a_A e_A \in \mathbb{R}_n$, its Clifford conjugate, denoted by $\overline{a}$, is defined as
\begin{align*}
	\overline{a} = \sum_A a_A \overline{e_A},
\end{align*}
where $\overline{e_{j_1} \cdots e_{j_r}} = \overline{e_{j_1}} \cdots \overline{e_{j_r}}$ and $\overline{e_j} = -e_j$ for $1 \leqslant j \leqslant n$, while $\overline{e_0} = e_0 = 1$.

The second involution is the \textbf{Clifford reversion}, which is defined for an element $a = \sum_A a_A e_A \in \mathbb{R}_n$ as
\begin{align*}
	\tilde{a} = \sum_A a_A \widetilde{e_A},
\end{align*}
where $\widetilde{e_{j_1} \cdots e_{j_r}} = e_{j_r} \cdots e_{j_1}$, and for any $a, b \in \mathbb{R}_n$, we have $\widetilde{ab} = \tilde{b} \tilde{a}$. This means that the reversion operation reverses the order of multiplication.
\par
The norm of an element $a \in \mathbb{R}_n$ is defined as $\left| a \right| = [a \overline{a}]_0 = \left( \sum_A |a_A|^2 \right)^{\frac{1}{2}}$. For a paravector $x \neq 0$, its norm is $\left| x \right| = (x \bar{x})^{\frac{1}{2}}$, and its inverse is given by $x^{-1} = \bar{x} \left| x \right|^{-2}$.

\subsection{Monogenic and slice monogenic functions}
Now, we introduce important concepts and notations associated with monogenic and slice monogenic functions. For further details, see \cite{Br2,Co4,De,Ge1}.

\begin{definition}[Monogenic functions]
	Let $\Omega \subset \mathbb{R}^{n+1}$ be a domain, and $f: \Omega \longrightarrow \mathbb{R}_n$ is a differentiable Clifford-valued function. The function $f$ is called \emph{left monogenic} on $\Omega$ if it satisfies the generalized Cauchy-Riemann equation
	\begin{align*}
		Df(x) = \sum_{i=0}^n e_i \frac{\partial f}{\partial x_i}(x) = 0, \quad \text{for all } x \in \Omega.
	\end{align*}
	Since the multiplications of Clifford numbers are not commutative in general, there is a similar definition for \emph{right monogenic} in $\Omega$.
	\end{definition}
	 The operator
\begin{align*}
	D = \sum_{i=0}^n e_i \frac{\partial}{\partial x_i} = \sum_{i=0}^n e_i \partial_{x_i} = \partial_{x_0} + \partial_{\underline{x}},
\end{align*}
where $D$ is the generalized Cauchy-Riemann operator (or Weyl operator), and $\partial_{\underline{x}}$ represents the classical Dirac operator in $\mathbb{R}^{n}$. 
\par
A central observation is that any non-real paravector in $\mathbb{R}^{n+1}$ can be written as
\begin{align*}
	x = x_0 + \sum_{i=1}^n x_i e_i = x_0 + \underline{x} = x_0 + r \omega,
\end{align*}
where $r = |\underline{x}| = \left( \sum_{i=1}^n x_i^2 \right)^{1/2}$ and $\omega = \frac{\underline{x}}{|\underline{x}|}$ is a uniquely defined unit vector resembling the classical imaginary unit. This implies that
\begin{align*}
	\omega \in \mathbb{S}^{n-1} = \left\{ x \in \mathbb{R}^{n+1}: x^2 = -1 \right\}.
\end{align*}
When $x$ is real, $r = 0$, and for every $\omega \in \mathbb{S}^{n-1}$, we have $x = x + \omega \cdot 0$.

\begin{definition}[Slice monogenic functions] 
	Let $\Omega \subset \mathbb{R}^{n+1}$ be a domain, and $f: \Omega \to \mathbb{R}_n$. The function $f$ is called \emph{(left) slice monogenic} if, for each direction $\omega \in \mathbb{S}^{n-1}$, the restriction of $f$ to the subset $\Omega_{\omega} = \Omega \cap (\mathbb{R} \oplus \omega \mathbb{R}) \subseteq \mathbb{R}^2$, denoted by $f_{\omega}$, is holomorphic. This means that $f_{\omega}$ has continuous partial derivatives and satisfies the following condition
	\begin{align*}
		\frac{1}{2}\left( \partial_{x_0} + \omega \partial_r \right) f_{\omega}(x_0 + r \omega) = 0
	\end{align*}
	for any $x_0 + r \omega \in \Omega_{\omega}$.
\end{definition}
\par
Another approach to defining slice monogenic functions, introduced by Ghiloni and Perotti \cite{Gh1} in 2011, uses the concept of stem functions. This framework has been widely explored in studies of slice monogenic functions, such as \cite{Di1, Gh2, Pe}. In the following section, we will review this approach for the theory of generalized partial-slice monogenic functions.

\section{Generalized partial-slice monogenic functions}
In \cite{Xu1,Xu2}, the authors introduced the concept of generalized partial-slice monogenic functions, extending the theory of slice monogenic functions. This new framework generalizes both the classical Clifford analysis and the theory of slice monogenic functions. In this section, we review some needed definitions and properties for the rest of the article. For a comprehensive discussion, see \cite{Xu2,Xu1}.
\par
Let $p$ and $q$ be non-negative integers, with $p,q > 0$. We consider functions $f : \Omega \to \mathbb{R}_{p+q}$, where $\Omega \subset \mathbb{R}^{p+q+1}$ is a domain. An element $\boldsymbol{x} \in \mathbb{R}^{p+q+1} = \mathbb{R}^{p+1} \oplus \mathbb{R}^q$ can be represented as a paravector in $\mathbb{R}_{p+q}$ as
\begin{align*}
	\boldsymbol{x} = \boldsymbol{x}_p + \underline{\boldsymbol{x}}_q \in \mathbb{R}^{p+1} \oplus \mathbb{R}^q, \quad \boldsymbol{x}_p = \sum_{i=0}^p x_i e_i, \quad \underline{\boldsymbol{x}}_q = \sum_{i=p+1}^{p+q} x_i e_i.
\end{align*}
Here, we define the generalized Cauchy-Riemann operator and the Euler operator as 
\begin{align*}
	D_{\boldsymbol{x}} &= \sum_{i=0}^{p+q} e_i \partial_{x_i} = \sum_{i=0}^p e_i \partial_{x_i} + \sum_{i=p+1}^{p+q} e_i \partial_{x_i} = D_{\boldsymbol{x}_p} + D_{\underline{\boldsymbol{x}}_q}, \\
	\mathbb{E}_{\boldsymbol{x}} &= \sum_{i=0}^{p+q} x_i \partial_{x_i} = \sum_{i=0}^p x_i \partial_{x_i} + \sum_{i=p+1}^{p+q} x_i \partial_{x_i} = \mathbb{E}_{\boldsymbol{x}_p} + \mathbb{E}_{\underline{\boldsymbol{x}}_q}.
\end{align*}
In these expressions, the operators $D_{\boldsymbol{x}_p}$ and $D_{\underline{\boldsymbol{x}}_q}$ act on the components $\boldsymbol{x}_p$ and $\underline{\boldsymbol{x}}_q$, respectively, and similarly for the Euler operators.

We denote the unit sphere in $\mathbb{R}^q$ by $\mathbb{S}$, consisting of elements $\underline{\boldsymbol{x}}_q = \sum_{i=p+1}^{p+q} x_i e_i$ that satisfy the condition
\begin{align*}
	\mathbb{S} = \left\{ \underline{\boldsymbol{x}}_q : {\underline{\boldsymbol{x}}_q}^2 = -1 \right\} = \left\{ \underline{\boldsymbol{x}}_q = \sum_{i=p+1}^{p+q} x_i e_i : \sum_{i=p+1}^{p+q} x_i^2 = 1 \right\}.
\end{align*}
For any non-zero vector $\underline{\boldsymbol{x}}_q$, there exists a unique $r \in \mathbb{R}^+$ and a unit vector $\underline{\omega} \in \mathbb{S}$ such that the vector $\underline{\boldsymbol{x}}_q$ can be written as
\begin{align*}
	\underline{\boldsymbol{x}}_q = r \underline{\omega}, \quad \text{where} \quad r = |\underline{\boldsymbol{x}}_q|, \quad \underline{\omega} = \frac{\underline{\boldsymbol{x}}_q}{|\underline{\boldsymbol{x}}_q|}.
\end{align*}
If $\underline{\boldsymbol{x}}_q = 0$, we define $r = 0$ and $\underline{\omega}$ is undefined. In this case, $\underline{\boldsymbol{x}}_q$ is the zero vector, which can be written as $\boldsymbol{x}_p + \underline{\omega} \cdot 0$ for any $\underline{\omega} \in \mathbb{S}$.

For an open set $\Omega \subset \mathbb{R}^{p+q+1}$, we introduce the notation
\begin{align*}
	\Omega_{\underline{\omega}} := \Omega \cap \left( \mathbb{R}^{p+1} \oplus \underline{\omega} \mathbb{R} \right) \subseteq \mathbb{R}^{p+2},
\end{align*}
which denotes the intersection of $\Omega$ with the subspace of $\mathbb{R}^{p+2}$ spanned by $\mathbb{R}^{p+1}$ and the line through the origin in the direction of $\underline{\omega}$. To develop the theory of generalized partial-slice monogenic functions, additional constraints on the domains are necessary.

\begin{definition}
	Let $\Omega$ be a domain in $\mathbb{R}^{p+q+1}$. The domain $\Omega$ is said to be \emph{partially symmetric} with respect to $\mathbb{R}^{p+1}$ (or \emph{$p$-symmetric} for short) if, for any $\boldsymbol{x}_p \in \mathbb{R}^{p+1}$, $r \in \mathbb{R}^+$, and $\underline{\omega} \in \mathbb{S}$, the following condition holds
	\begin{align*}
		\boldsymbol{x} = \boldsymbol{x}_p + r \underline{\omega} \in \Omega \quad \Longrightarrow \quad \left[ \boldsymbol{x} \right] := \boldsymbol{x}_p + r \mathbb{S} = \left\{ \boldsymbol{x}_p + r \underline{\omega} : \underline{\omega} \in \mathbb{S} \right\} \subseteq \Omega.
	\end{align*}
\end{definition}
Next, we introduce the concept of stem functions, which are fundamental to the study of generalized partial-slice functions.
 
\begin{definition}
	Let $F: D \longrightarrow \mathbb{R}_{p+q} \otimes_{\mathbb{R}} \mathbb{C}$ be a function defined on an open set $D \subseteq \mathbb{R}^{p+2}$. We say that $F$ is  \emph{stem function} if it is invariant under the reflection of the $(p+2)$-th variable, and its $\mathbb{R}_{p+q}$-valued components, $F_1$ and $F_2$ (where $F = F_1 + iF_2$), satisfy the following conditions
	\begin{align}\label{Stem Function Costituent}
		F_1(\boldsymbol{x}_p, -r) = F_1(\boldsymbol{x}_p, r), \quad F_2(\boldsymbol{x}_p, -r) = -F_2(\boldsymbol{x}_p, r), \quad (\boldsymbol{x}_p, r) \in D.
	\end{align}
	Given a stem function $F$, it induces a (left) generalized partial-slice function $f = \mathcal{I}(F)$ from $\Omega_D$ to $\mathbb{R}_{p+q}$, defined as
	\begin{align*}
		f(\boldsymbol{x}) := F_1(\boldsymbol{x}') + \underline{\omega} F_2(\boldsymbol{x}'), \quad \boldsymbol{x}' = (\boldsymbol{x}_p, r), \quad \boldsymbol{x} = \boldsymbol{x}_p + r \underline{\omega} \in \mathbb{R}^{p+q+1}, \quad \underline{\omega} \in \mathbb{S}.
	\end{align*}
\end{definition}
\begin{definition}
	Let $D \subseteq \mathbb{R}^{p+2}$ be a domain that is invariant under the reflection of the $(p+2)$-th variable. The \emph{p-symmetric completion} of $D$, denoted by $\Omega_D \subset \mathbb{R}^{p+q}$, is defined as
	\begin{align*}
		\Omega_D = \bigcup_{\underline{\omega} \in \mathbb{S}} \left\{ x_p + r \underline{\omega} : \exists x_p \in \mathbb{R}^{p+1}, r \geq 0, (x_p, r) \in D \right\}.
	\end{align*}
\end{definition}

It is important to note that a domain $\Omega \subset \mathbb{R}^{p+q+1}$ is $p$-symmetric if and only if there exists a corresponding domain $D \subset \mathbb{R}^{p+2}$ such that $\Omega = \Omega_D$. Throughout this paper, we will use $\Omega_D$ to refer to a $p$-symmetric domain in $\mathbb{R}^{p+q+1}$.
\par
We define the set of all generalized partial-slice functions induced on $\Omega_D$ as
$$ \mathcal{G}\mathcal{S}(\Omega_D) := \left\{ f = \mathcal{I}(F) : F \text{ is a stem function on } D \right\}. $$
\begin{definition}
	Let $f \in \mathcal{G}\mathcal{S}(\Omega_D)$. We say that $f$ is \emph{generalized partial-slice monogenic of type $(p,q)$} if its corresponding stem function $F = F_1 + iF_2$ satisfies the generalized Cauchy-Riemann equations
	$$
	\begin{cases}  
		D_{\boldsymbol{x}_p} F_1 - \partial_r F_2 = 0, \\  
		\overline{D_{\boldsymbol{x}_p}} F_2 - \partial_r F_1 = 0.
	\end{cases}
	$$
\end{definition}
Similar to the case of slice functions, there exists a representation formula for generalized partial-slice functions of type $(p,q)$, as described below.
\begin{theorem}[Representation Formula]\label{*Representation Formula*}
	\cite{Xu2} Let $f\in \mathcal{G} \mathcal{S} \left( \Omega _D \right)$. Then it holds that, for every $\boldsymbol{x}=\boldsymbol{x}_p+r\underline{\omega}\in \Omega _D$ with $\underline{\omega}\in \mathbb{S}$,
	\begin{align}
		f\left( \boldsymbol{x} \right)
		=&\left( \underline{\omega}-\underline{\omega}_2 \right) \left( \underline{\omega}_1-\underline{\omega}_2 \right) ^{-1}f\left( \boldsymbol{x}_p+r\underline{\omega}_1 \right)\nonumber\\
& -\left( \underline{\omega}-\underline{\omega}_1 \right) \left( \underline{\omega}_1-\underline{\omega}_2 \right) ^{-1}f\left( \boldsymbol{x}_p+r\underline{\omega}_2 \right) ,
	\end{align}
	for all $\underline{\omega}_1\ne \underline{\omega}_2\in \mathbb{S} $. In particular, if $\underline{\omega}_1=-\underline{\omega}_2 =\underline{\eta}\in \mathbb{S} $, we have
	\begin{align*}
		f\left( \boldsymbol{x} \right) &=\frac{1}{2}\left( 1-\underline{\omega}\underline{\eta} \right) f\left( \boldsymbol{x}_p+r\underline{\eta} \right) +\frac{1}{2}\left( 1+\underline{\omega}\underline{\eta} \right) f\left( \boldsymbol{x}_p-r\underline{\eta} \right)\\
		&=\frac{1}{2}\left( f\left( \boldsymbol{x}_p+r\underline{\eta} \right) +f\left( \boldsymbol{x}_p-r\underline{\eta} \right) \right) +\frac{1}{2}\underline{\omega}\underline{\eta}\left( f\left( \boldsymbol{x}_p-r\underline{\eta} \right) -f\left( \boldsymbol{x}_p+r\underline{\eta} \right) \right) .
	\end{align*}
\end{theorem}
Recall that the Cauchy kernel for monogenic functions in $\mathbb{R}^{p+2}$ is expressed as
\begin{align*}
	E(\boldsymbol{x}) = \frac{1}{\sigma_{p+1}} \frac{\overline{\boldsymbol{x}}}{|\boldsymbol{x}|^{p+2}}, \quad x \in \mathbb{R}^{p+2} \setminus \{0\},
\end{align*}
where $\sigma_{p+1} = 2 \frac{\Gamma^{p+2}\left( \frac{1}{2} \right)}{\Gamma \left( \frac{p+2}{2} \right)}$ represents the surface area of the unit sphere in $\mathbb{R}^{p+2}$. Given the representation formula for generalized partial-slice monogenic functions in Theorem \ref{*Representation Formula*}, we naturally define the generalized partial-slice Cauchy kernel as follows.
\begin{definition}
	For $\boldsymbol{y} \in \mathbb{R}^{p+q+1}$, the left generalized partial-slice Cauchy kernel $\mathcal{E}_{\boldsymbol{y}}(\cdot)$ is defined as
	\begin{eqnarray}
		\mathcal{E}_{\boldsymbol{y}}(\boldsymbol{x}) = \frac{1}{2} \left( 1 - \underline{\omega} \underline{\eta} \right) E_{\boldsymbol{y}}\left( \boldsymbol{x}_p + r \underline{\eta} \right) + \frac{1}{2} \left( 1 + \underline{\omega} \underline{\eta} \right) E_{\boldsymbol{y}}\left( \boldsymbol{x}_p - r \underline{\eta} \right),
	\end{eqnarray}
	where $\underline{\omega}$ and $\underline{\eta}$ are defined as in Theorem \ref{*Representation Formula*}.
\end{definition}
Using the Cauchy-Pompeiu formula for monogenic functions along with the representation formula given in Theorem \ref{*Representation Formula*}, Cauchy-Pompeiu formula for partial-slice monogenic functions was derived in \cite{Xu2}. More specifically,
\begin{theorem}[Cauchy-Pompeiu formula]\label{CPF} \cite{Xu2} Let $f=\mathcal{I} \left( F \right) \in \mathcal{G} \mathcal{S} \left( \Omega _D \right) $ with its stem function $F\in C^1\left( \overline{D} \right)$ and set $\Omega =\Omega _D$. If $U$ is a domain in $\mathbb{R}^{p+q+1}$ such that $U_{\underline{\eta}}\subset \Omega _{\underline{\eta}}$ is a bounded domain in $\mathbb{R}^{p+2}$ with smooth boundary $\partial U_{\underline{\eta}}\subset \Omega _{\underline{\eta}}$ for some $\underline{\eta}\in \mathbb{S}$, then for any $\boldsymbol{x}\in U$, we have
	\begin{eqnarray}
		f\left( \boldsymbol{x} \right) =\int_{\partial U_{\underline{\eta}}}{\mathcal{E} _{\boldsymbol{y}}\left( \boldsymbol{x} \right) n\left( \boldsymbol{y} \right) f\left( \boldsymbol{y} \right) dS_{\underline{\eta}} \left( \boldsymbol{y} \right)}-\int_{U_{\underline{\eta}}}{\mathcal{E} _{\boldsymbol{y}}\left( \boldsymbol{x} \right) \left( D_{\underline{\eta}}f \right) \left( \boldsymbol{y} \right) d\sigma  _{\underline{\eta}}\left( \boldsymbol{y} \right)},
	\end{eqnarray}
	where $n\left( \boldsymbol{y} \right) =\sum_{i=0}^p{n_i\left( \boldsymbol{y} \right) e_i}+n_{p+1}\left( \boldsymbol{y} \right) \underline{\eta}$ is the unit exterior normal vector to $\partial U_{\underline{\eta}}$ at $\boldsymbol{y}$, $dS_{\underline{\eta}} $ and $d\sigma  _{\underline{\eta}}$ stand for the classical Lebesgue surface element and volume element in $\mathbb{R}^{p+2}$, respectively.
\end{theorem}
\par
Additionally, the global differential operator with non-constant coefficients for a $C^1$ function $f : \Omega \longrightarrow \mathbb{R}_{p+q}$ is given by
\begin{align*}
	\bar{\vartheta}f(\boldsymbol{x}) = D_{\boldsymbol{x}_p} f(\boldsymbol{x}) + \frac{\underline{\boldsymbol{x}}_q}{|\underline{\boldsymbol{x}}_q|^2} \mathbb{E}_{\underline{\boldsymbol{x}}_q} f(\boldsymbol{x}).
\end{align*}
Due to the singularities arising from the term $|\underline{\boldsymbol{x}}_q|^2$ in the operator, a more careful treatment is required. Hence, for the remainder of this paper, we introduce the notation $\mathbb{R}_*^{p+q+1} := \mathbb{R}^{p+q+1} \setminus \mathbb{R}^{p+1}$.
\par
A straightforward calculation leads to the following results.
\begin{proposition}\cite{Xu2}
	Let $\Omega$ be a domain in $\mathbb{R} _{*}^{p+q+1}$. For the $C^1$ function $f:\Omega \longrightarrow \mathbb{R} _{p+q}$, it holds that
	\begin{enumerate}
		\item $\bar{\vartheta}f\left( \boldsymbol{x} \right) =D_{\underline{\omega}}f\left( \boldsymbol{x} \right),\ \boldsymbol{x}=\boldsymbol{x}_p+r\underline{\omega}$,
		\item $f\in ker\bar{\vartheta}\Leftrightarrow f\in kerG_{\boldsymbol{x}}$.
	\end{enumerate}
\end{proposition}

\section{Norm estimates of Teodorescu transform}
In \cite{Hu}, the Cauchy kernel definition for generalized partial-slice monogenic functions is presented as follows
\begin{align*}
	K_{\boldsymbol{y}}\left( \boldsymbol{x} \right) =\frac{\mathcal{E} _{\boldsymbol{y}}\left( \boldsymbol{x} \right)}{\sigma _{q-1}\left| \underline{\boldsymbol{y}}_q \right|^{q-1}}
\end{align*}
where $\sigma_{q-1}$ is the area of the $\left( q-1 \right) \text{-}$sphere $\mathbb{S}$.
\par
The Cauchy-Pompeiu formula mentioned earlier is only applicable to slice domains $\Omega_{\underline{\eta}}$. However, by employing the methods presented in \cite[Theorem 3.5, 3.6]{Di1}, we can readily derive the Cauchy-Pompeiu formula and the Cauchy integral formula on $\Omega_D$, as shown below.

\begin{theorem}[Cauchy-Pompeiu Formula]\label{CPFG}
	Let $\Omega _D\subset \mathbb{R} _{*}^{p+q+1}$ be a bounded domain as previously defined, and let $f \in \ker \bar{\vartheta}$ be a generalized partial-slice function. Suppose $U$ is a domain in $\mathbb{R}^{p+q+1}$ such that $U_D \subset \Omega_D$ is a bounded domain in $\mathbb{R}^{p+2}$ with a smooth boundary $\partial U_D \subset \Omega_D$. Then, for any $\boldsymbol{x} \in U$, the following holds
	\begin{align*}
		f(\boldsymbol{x}) = \int_{\partial U_D} K_{\boldsymbol{y}}(\boldsymbol{x}) \, n(\boldsymbol{y}) \, f(\boldsymbol{y}) \, dS(\boldsymbol{y}) - \int_{U_D} K_{\boldsymbol{y}}(\boldsymbol{x}) \, (\bar{\vartheta}f)(\boldsymbol{y}) \, d\sigma(\boldsymbol{y}),
	\end{align*}
	where $n(\boldsymbol{y})$ denotes the unit outward normal vector to $\partial U_D$ at $\boldsymbol{y}$, and $dS$ and $d\sigma$ represent the Lebesgue surface and volume element in $\mathbb{R}^{p+2}$, respectively.
\end{theorem}

\begin{theorem}[Cauchy Integral Formula]
    Let $\Omega _D\subset \mathbb{R} _{*}^{p+q+1}$ be a bounded domain as previously defined, and let $f \in \ker \bar{\vartheta}$ be a slice function. Then, for any $\boldsymbol{x} \in \Omega_D$, the following identity holds
	\begin{align} \label{CIF}
		\int_{\partial \Omega_D} K_{\boldsymbol{y}}(\boldsymbol{x}) \, n(\boldsymbol{y}) \, f(\boldsymbol{y}) \, d\sigma(\boldsymbol{y}) = f(\boldsymbol{x}),
	\end{align}
	where $n(\boldsymbol{y})$ is the unit outward normal vector to the boundary $\partial \Omega_D$ at $\boldsymbol{y}$.
\end{theorem}
\par

In \cite{Hu}, we further define the following operators
\begin{align*}
	T_{\Omega_D}f(\boldsymbol{x}) &= -\int_{\Omega_D} K_{\boldsymbol{y}}(\boldsymbol{x}) f(\boldsymbol{y}) \, d\sigma(\boldsymbol{y}), \\
	F_{\partial \Omega_D}f(\boldsymbol{x}) &= \int_{\partial \Omega_D} K_{\boldsymbol{y}}(\boldsymbol{x}) n(\boldsymbol{y}) f(\boldsymbol{y}) \, dS(\boldsymbol{y}).
\end{align*}
 Using these definitions, the Cauchy-Pompeiu formula in Theorem \ref{CPFG} can be expressed as
\begin{align*}
	F_{\partial \Omega_D}f(\boldsymbol{x}) + T_{\Omega_D}(\bar{\vartheta}f)(\boldsymbol{x}) = f(\boldsymbol{x}), \quad \boldsymbol{x} \in \Omega_D.
\end{align*}
Here, $T_{\Omega_D}$ is referred to as the \textit{Teodorescu transform}. For functions with compact support in $\Omega_D$, we have $F_{\partial \Omega_D}f(\boldsymbol{x}) = 0$, which simplifies the above equation to
\begin{align*}
	T_{\Omega_D}(\bar{\vartheta}f)(\boldsymbol{x}) = f(\boldsymbol{x}).
\end{align*}
This indicates that $T_{\Omega_D}$ acts as a left inverse of $\bar{\vartheta}$ for compactly supported functions.

Next, we discuss the existence of $T_{\Omega_D}f$ and provide a norm estimate for $T_{\Omega_D}$. To facilitate our analysis, we introduce the slice Teodorescu transform
\begin{align*}
	T_{\Omega _{\underline{\omega }}}f(\boldsymbol{x})=-\int_{\Omega _{\underline{\omega }}}{\mathcal{E} _{\boldsymbol{y}}(\boldsymbol{x})f(\boldsymbol{y})\,d\sigma _{\underline{\omega }}(\boldsymbol{y})}.
\end{align*}
 Additionally, we utilize a fundamental theorem from measure theory, which ensures the interchangeability of differentiation and integration, as a key tool in our arguments.

\begin{theorem}\label{interchange}
	\cite{Fo} Suppose that $f:X\times \left[ a,b \right] \longrightarrow \mathbb{C} ,\left( -\infty <a<b<+\infty \right) $ and $f\left( \cdot ,t \right) :X\longrightarrow \mathbb{C} $ is integrable for each $t\in \left[ a,b \right] $. Let 
	\begin{align*}
		F\left( t \right) =\int_X{f\left( x,t \right) d\mu \left( x \right)},
	\end{align*} and 
	\begin{enumerate}
		\item Suppose that $\frac{\partial f}{\partial t}$ exists;
		\item $\exists g\in L^1\left( \mu \right) $ such that $\left| \frac{\partial f}{\partial t}\left( x,t \right) \right|\leqslant g\left( x \right) $ for all $x$ and $t$.
	\end{enumerate}
	Then $F$ is differentiable and 
	\begin{align*}
		F^{\prime}\left( t \right) =\int_X{\frac{\partial f}{\partial t}\left( x,t \right) d\mu \left( x \right)}.
	\end{align*}
\end{theorem}
We also need the following well-known Stokes' Theorem for the conjugate Cauchy-Riemann operator $\partial _{\boldsymbol{x}_I}$ as follows.
\begin{theorem}
	\cite{Gu} Let $\Omega _I\subset \mathbb{C} _I$ be a domain with sufficiently smooth boundary $\partial \Omega $, and $f\left( \boldsymbol{x} \right) $, $g\left( \boldsymbol{x} \right) \in C^1\left( \overline{\Omega _I} \right) $. Then, we have 
	\begin{align*}
		\int_{\Omega _I}{\left( f\left( \boldsymbol{x} \right) \partial _{\boldsymbol{x}_I} \right) g\left( \boldsymbol{x} \right) +f\left( \boldsymbol{x} \right) \left( \partial _{\boldsymbol{x}_I}g\left( \boldsymbol{x} \right) \right) d\sigma _I\left( \boldsymbol{x} \right)}=\int_{\partial \Omega _I}{f\left( \boldsymbol{x} \right) \overline{d\boldsymbol{x}^*}g\left( \boldsymbol{x} \right)},
	\end{align*}
	where $\overline{d\boldsymbol{x}^*}=Id\boldsymbol{x}$ and $d\boldsymbol{x}$ is the line element on $\partial \Omega _I$.
\end{theorem}
Now, we introduce the derivatives of the slice Teodorescu transform as follows. This is important to obtain the result that $\bar{\vartheta}$ is the left inverse of the Teodorescu transform $T_{\Omega_D}$.
\begin{theorem}\label{Theta T=2f}
	Let $\Omega _D\subset \mathbb{R} _{*}^{p+q+1}$ be a bounded p-symmetric domain, $f\in C^1\left( \overline{\Omega _D} \right) $, then, for $\boldsymbol{x}\in \Omega _D$, we have 
	\begin{align*}
		\partial _{x_{p_i}}T_{\Omega _{\underline{\omega }}}f\left( \boldsymbol{x} \right) =\int_{\Omega _{\underline{\omega }}}{\left( \partial _{x_{p_i}}\mathcal{E} _{\boldsymbol{y}}\left( \boldsymbol{x} \right) \right) f\left( \boldsymbol{y} \right) d\sigma _{\underline{\omega }}\left( \boldsymbol{y} \right)}-e_i\left[ \alpha f\left( \boldsymbol{x}_{\underline{\omega }} \right) +\beta f\left( \boldsymbol{x}_{-\underline{\omega }} \right) \right],
	\end{align*}
	\begin{align*}
		&\partial _{x_{q_i}}T_{\Omega _{\underline{\omega }}}f\left( \boldsymbol{x} \right)\\
		=&\frac{x_{q_i}}{\left| \underline{\boldsymbol{x}}_q \right|}\left[ \int_{\Omega _{\underline{\omega }}}{\left( \frac{\partial}{\partial r}\mathcal{E} _{\boldsymbol{y}}\left( \boldsymbol{x} \right) \right) f\left( \boldsymbol{y} \right) d\sigma _{\underline{\omega }}\left( \boldsymbol{y} \right)}-\left( \alpha \underline{\omega }f\left( \boldsymbol{x}_{\underline{\omega }} \right) -\beta \underline{\omega }f\left( \boldsymbol{x}_{-\underline{\omega }} \right) \right) \right],
	\end{align*}
	where 
	\begin{align*}
		\alpha =\frac{1-\underline{\omega}\underline{\eta}}{2}, \beta =\frac{1+\underline{\omega}\underline{\eta}}{2}.
	\end{align*}
	In particular, if $f\in\mathcal{GS}(\Omega_D)$, we have 
	\begin{align*}
		\bar{\vartheta}T_{\Omega _{\underline{\omega}}}f\left( \boldsymbol{x} \right) =2f\left( \boldsymbol{x} \right).
	\end{align*}
\end{theorem}
The expression $T_{\Omega_{\underline{\omega}}}f$ represents a singular integral only when $\boldsymbol{x}$ lies within the domain $\Omega_{\underline{\omega}}$. Our method relies on the fact that $\boldsymbol{x}$ can be rewritten in terms of $\boldsymbol{x}_{\underline{\omega}}$ and $\boldsymbol{x}_{-\underline{\omega}}$ by the representation formula. This transformation gives rise to two singular integral operators, which should be considered as Cauchy's principal values.
\begin{proof}
	Firstly, we denote $\boldsymbol{x}=\boldsymbol{x}_p+\underline{\eta}r$, where $\underline{\eta}=\frac{\underline{\boldsymbol{x}}_q}{\left| \underline{\boldsymbol{x}}_q \right|}$ and $r=\left| \underline{\boldsymbol{x}}_q \right|$. Similarly, we denote $\boldsymbol{y}=\boldsymbol{y}_p+\underline{\omega }\tilde{r}$, where $\underline{\omega }=\frac{\underline{\boldsymbol{y}}_q}{\left| \underline{\boldsymbol{y}}_q \right|}$ and $\tilde{r}=\left| \underline{\boldsymbol{y}}_q \right|$. Let $\boldsymbol{x}_{\underline{\omega}}=\boldsymbol{x}_p+\underline{\omega}r$, and $D_{\boldsymbol{x}_{\underline{\omega }}}=D_{\boldsymbol{x}_p}+\underline{\omega }\partial _r=\sum_{i=0}^p{e_i\partial _{x_i}}+\underline{\omega }\partial _r$. We know that 
	\begin{align*}
		\mathcal{E} _{\boldsymbol{y}}\left( \boldsymbol{x} \right) &=\frac{1}{2}\left( 1-\underline{\omega}\underline{\eta} \right) E_{\boldsymbol{y}}\left( \boldsymbol{x}_{\underline{w}} \right) +\frac{1}{2}\left( 1+\underline{\omega}\underline{\eta} \right) E_{\boldsymbol{y}}\left( \boldsymbol{x}_{-\underline{w}} \right)\\
		&=\alpha E_{\boldsymbol{y}}\left( \boldsymbol{x}_{\underline{w}} \right) +\beta E_{\boldsymbol{y}}\left( \boldsymbol{x}_{-\underline{w}} \right) ,
	\end{align*}
	and 
	\begin{align*}
		\sigma _{p+1}E_{\boldsymbol{y}}\left( \boldsymbol{x}_{\underline{w}} \right) =\frac{\overline{\boldsymbol{y}-\boldsymbol{x}_{\underline{w}}}}{\left| \boldsymbol{y}-\boldsymbol{x}_{\underline{w}} \right|^{p+2}}=\frac{1}{p}D_{\boldsymbol{x}_{\underline{w}}}\frac{1}{\left| \boldsymbol{y}-\boldsymbol{x}_{\underline{w}} \right|^p}=-\frac{1}{p}D_{\boldsymbol{y}_{\underline{w}}}\frac{1}{\left| \boldsymbol{y}-\boldsymbol{x}_{\underline{w}} \right|^p}.
	\end{align*}
	Since $T_{\Omega _{\underline{\omega}}}f$ is a singular integral, which only makes sense as a Cauchy principal value, let $B_{\epsilon}=B\left( \boldsymbol{x}_{\underline{\omega}},\epsilon \right) \cup B\left( \boldsymbol{x}_{-\underline{\omega}},\epsilon \right) \subset \Omega _{\underline{\omega}}$ for a sufficiently small $\epsilon >0$. Then, we have 
	\begin{align*}
		&-\sigma _{p+1}T_{\Omega _{\underline{\omega}}}f\left( \boldsymbol{x} \right)  =\sigma _{p+1}\int_{\Omega _{\underline{\omega }}}{\mathcal{E} _{\boldsymbol{y}}\left( \boldsymbol{x} \right) f\left( \boldsymbol{y} \right) d\sigma _{\underline{\omega }}\left( \boldsymbol{y} \right)}\\
		=&\sigma _{p+1}\lim_{\epsilon \rightarrow 0} \int_{\Omega _{\underline{\omega}}\backslash B_{\epsilon}}{\mathcal{E} _{\boldsymbol{y}}\left( \boldsymbol{x} \right) f\left( \boldsymbol{y} \right) d\sigma _{\underline{\omega}}\left( \boldsymbol{y} \right)}\\
		=&\sigma _{p+1}\lim_{\epsilon \rightarrow 0} \int_{\Omega _{\underline{\omega}}\backslash B_{\epsilon}}{\left( \alpha E_{\boldsymbol{y}}\left( \boldsymbol{x}_{\underline{w}} \right) +\beta E_{\boldsymbol{y}}\left( \boldsymbol{x}_{-\underline{w}} \right) \right) f\left( \boldsymbol{y} \right) d\sigma _{\underline{\omega}}\left( \boldsymbol{y} \right)}\\
		=&-\frac{1}{p}\lim_{\epsilon \rightarrow 0} \int_{\Omega _{\underline{\omega }}\backslash B_{\epsilon}}{\left[ \left( \alpha \frac{1}{\left| \boldsymbol{y}-\boldsymbol{x}_{\underline{w}} \right|^p}+\beta \frac{1}{\left| \boldsymbol{y}-\boldsymbol{x}_{-\underline{w}} \right|^p} \right) D_{\boldsymbol{y}_{\underline{\omega }}} \right] f\left( \boldsymbol{y} \right) d\sigma _{\underline{\omega }}\left( \boldsymbol{y} \right)}\\
		=&-\frac{1}{p}\lim_{\epsilon \rightarrow 0} \left[ -\int_{\Omega _{\underline{\omega }}\backslash B_{\epsilon}}{\left( \alpha \frac{1}{\left| \boldsymbol{y}-\boldsymbol{x}_{\underline{w}} \right|^p}+\beta \frac{1}{\left| \boldsymbol{y}-\boldsymbol{x}_{-\underline{w}} \right|^p} \right) \left( D_{\boldsymbol{y}_{\underline{\omega}}}f\left( \boldsymbol{y} \right) \right) d\sigma _{\underline{\omega}}\left( \boldsymbol{y} \right)} \right] \\
		& -\frac{1}{p}\lim_{\epsilon \rightarrow 0} \left[ \left( \int_{\partial \Omega _{\underline{\omega }}}{-\int_{\partial B_{\epsilon}}} \right) \left( \alpha \frac{1}{\left| \boldsymbol{y}-\boldsymbol{x}_{\underline{w}} \right|^p}+\beta \frac{1}{\left| \boldsymbol{y}-\boldsymbol{x}_{-\underline{w}} \right|^p} \right) \overline{d\boldsymbol{y}^*}f\left( \boldsymbol{y} \right) \right] 
		\\
		=&-\frac{1}{p}\lim_{\epsilon \rightarrow 0} \left[ -\int_{\Omega _{\underline{\omega }}}{\left( \alpha \frac{1}{\left| \boldsymbol{y}-\boldsymbol{x}_{\underline{w}} \right|^p}+\beta \frac{1}{\left| \boldsymbol{y}-\boldsymbol{x}_{-\underline{w}} \right|^p} \right) \left( D_{\boldsymbol{y}_{\underline{\omega}}}f\left( \boldsymbol{y} \right) \right) d\sigma _{\underline{\omega}}\left( \boldsymbol{y} \right)} \right] \\
		& -\frac{1}{p}\int_{\partial \Omega _{\underline{\omega }}}{\left( \alpha \frac{1}{\left| \boldsymbol{y}-\boldsymbol{x}_{\underline{w}} \right|^p}+\beta \frac{1}{\left| \boldsymbol{y}-\boldsymbol{x}_{-\underline{w}} \right|^p} \right)}\overline{d\boldsymbol{y}^*}f\left( \boldsymbol{y} \right) .
	\end{align*}
	By Theorem $\ref{interchange}$, differentiation and integration can be interchanged. Indeed, since $f\in C^1\left( \overline{\Omega _D} \right) $, which implies that $f$ is bounded. Further, the homogeneity of
	\begin{align*}
		\partial _{x_{p_i}}\left[ \alpha \frac{1}{\left| \boldsymbol{y}-\boldsymbol{x}_{\underline{\omega }} \right|^p}+\beta \frac{1}{\left| \boldsymbol{y}-\boldsymbol{x}_{-\underline{\omega }} \right|^p} \right] =-p\left[ \alpha \frac{x_{p_i}-y_{p_i}}{\left| \boldsymbol{y}-\boldsymbol{x}_{\underline{\omega }} \right|^{p+2}}+\beta \frac{x_{p_i}-y_{p_i}}{\left| \boldsymbol{y}-\boldsymbol{x}_{-\underline{\omega }} \right|^{p+2}} \right] 
	\end{align*}
	suggests that it is integrable with respect to $\boldsymbol{x}$, which means that the two conditions of Theorem \ref{interchange} are satisfied. Therefore, we get 
	\begin{align}\label{partial xpi}
		&\sigma _{p+1}\partial _{x_{p_i}}T_{\Omega _{\underline{\omega }}}f\left( \boldsymbol{x} \right)\nonumber \\
		=&\int_{\Omega _{\underline{\omega }}}{\left( \alpha \frac{x_{p_i}-y_{p_i}}{\left| \boldsymbol{y}-\boldsymbol{x}_{\underline{\omega }} \right|^{p+2}}+\beta \frac{x_{p_i}-y_{p_i}}{\left| \boldsymbol{y}-\boldsymbol{x}_{-\underline{\omega }} \right|^{p+2}} \right) \left( D_{\boldsymbol{y}_{\underline{\omega }}}f\left( \boldsymbol{y} \right) \right) d\sigma _{\underline{\omega }}\left( \boldsymbol{y} \right)}\nonumber \\
		&-\int_{\partial \Omega _{\underline{\omega }}}{\left( \alpha \frac{x_{p_i}-y_{p_i}}{\left| \boldsymbol{y}-\boldsymbol{x}_{\underline{\omega }} \right|^{p+2}}+\beta \frac{x_{p_i}-y_{p_i}}{\left| \boldsymbol{y}-\boldsymbol{x}_{-\underline{\omega }} \right|^{p+2}} \right) \overline{d\boldsymbol{y}^*}f\left( \boldsymbol{y} \right)}.
	\end{align}
	And then, using Gauss theorem, we can get
	\begin{align*}
		&\quad \left( \int_{\partial \Omega _{\underline{\omega }}}{-\int_{\partial B_{\epsilon}}{}} \right) \left( \alpha \frac{x_{p_i}-y_{p_i}}{\left| \boldsymbol{y}-\boldsymbol{x}_{\underline{\omega }} \right|^{p+2}}+\beta \frac{x_{p_i}-y_{p_i}}{\left| \boldsymbol{y}-\boldsymbol{x}_{-\underline{\omega }} \right|^{p+2}} \right) \overline{d\boldsymbol{y}^*}f\left( \boldsymbol{y} \right) \\
		&=\int_{\partial \Omega _{\underline{\omega }}\backslash B_{\epsilon}}{\left[ \left( \alpha \frac{x_{p_i}-y_{p_i}}{\left| \boldsymbol{y}-\boldsymbol{x}_{\underline{\omega }} \right|^{p+2}}+\beta \frac{x_{p_i}-y_{p_i}}{\left| \boldsymbol{y}-\boldsymbol{x}_{-\underline{\omega }} \right|^{p+2}} \right) D_{\boldsymbol{y}_{\underline{\omega }}} \right] f\left( \boldsymbol{y} \right)}\\
		&\quad \quad \quad \quad \quad +\left( \alpha \frac{x_{p_i}-y_{p_i}}{\left| \boldsymbol{y}-\boldsymbol{x}_{\underline{\omega }} \right|^{p+2}}+\beta \frac{x_{p_i}-y_{p_i}}{\left| \boldsymbol{y}-\boldsymbol{x}_{-\underline{\omega }} \right|^{p+2}} \right) \left( D_{\boldsymbol{y}_{\underline{\omega }}}f\left( \boldsymbol{y} \right) \right) d\sigma _{\underline{\omega }}\left( \boldsymbol{y} \right).
	\end{align*} 
	Hence, we have 
	\begin{align*}
		&\sigma _{p+1}\partial _{x_{p_i}}T_{\Omega _{\underline{\omega }}}f\left( \boldsymbol{x} \right)\\
		=&\lim_{\epsilon \rightarrow 0} \int_{B_{\epsilon}}{\left( \alpha \frac{x_{p_i}-y_{p_i}}{\left| \boldsymbol{y}-\boldsymbol{x}_{\underline{\omega }} \right|^{p+2}}+\beta \frac{x_{p_i}-y_{p_i}}{\left| \boldsymbol{y}-\boldsymbol{x}_{-\underline{\omega }} \right|^{p+2}} \right) \left( D_{\boldsymbol{y}_{\underline{\omega }}}f\left( \boldsymbol{y} \right) \right) d\sigma _{\underline{\omega }}\left( \boldsymbol{y} \right)}\\
		&-\int_{\Omega _{\underline{\omega }}\backslash B_{\epsilon}}{\left[ \left( \alpha \frac{x_{p_i}-y_{p_i}}{\left| \boldsymbol{y}-\boldsymbol{x}_{\underline{\omega }} \right|^{p+2}}+\beta \frac{x_{p_i}-y_{p_i}}{\left| \boldsymbol{y}-\boldsymbol{x}_{-\underline{\omega }} \right|^{p+2}} \right) D_{\boldsymbol{y}_{\underline{\omega }}} \right] f\left( \boldsymbol{y} \right) d\sigma _{\underline{\omega }}\left( \boldsymbol{y} \right)}\\
		&-\int_{\partial B_{\epsilon}}{\left( \alpha \frac{x_{p_i}-y_{p_i}}{\left| \boldsymbol{y}-\boldsymbol{x}_{\underline{\omega }} \right|^{p+2}}+\beta \frac{x_{p_i}-y_{p_i}}{\left| \boldsymbol{y}-\boldsymbol{x}_{-\underline{\omega }} \right|^{p+2}} \right) \overline{d\boldsymbol{y}^*}f\left( \boldsymbol{y} \right)}.
	\end{align*}
	From the homogeneity of $\frac{x_{p_i}-y_{p_i}}{\left| \boldsymbol{y}-\boldsymbol{x}_{\underline{\omega }} \right|^{p+2}}$ and $\frac{x_{p_i}-y_{p_i}}{\left| \boldsymbol{y}-\boldsymbol{x}_{-\underline{\omega }} \right|^{p+2}}$, on the one hand, one can easily show that 
	\begin{align*}
		\lim_{\epsilon \rightarrow 0} \int_{B_{\epsilon}}{\left( \alpha \frac{x_{p_i}-y_{p_i}}{\left| \boldsymbol{y}-\boldsymbol{x}_{\underline{\omega }} \right|^{p+2}}+\beta \frac{x_{p_i}-y_{p_i}}{\left| \boldsymbol{y}-\boldsymbol{x}_{-\underline{\omega }} \right|^{p+2}} \right) \left( D_{\boldsymbol{y}_{\underline{\omega }}}f\left( \boldsymbol{y} \right) \right) d\sigma _{\underline{\omega }}\left( \boldsymbol{y} \right)}=0,
	\end{align*}
	and 
	\begin{align*}
		&\frac{x_{p_i}-y_{p_i}}{\left| \boldsymbol{y}-\boldsymbol{x}_{\pm \underline{\omega}} \right|^{p+2}}D_{\boldsymbol{y}_{\underline{\omega }}}=\frac{1}{-p}\partial _{x_{p_i}}\frac{1}{\left| \boldsymbol{y}-\boldsymbol{x}_{\pm \underline{\omega}} \right|^p}D_{\boldsymbol{y}_{\underline{\omega }}}\\
		=&\frac{1}{-p}\partial _{x_{p_i}}D_{\boldsymbol{y}_{\underline{\omega }}}\frac{1}{\left| \boldsymbol{y}-\boldsymbol{x}_{\pm \underline{\omega}} \right|^p}=\partial _{x_{p_i}}\frac{\overline{\boldsymbol{y}-\boldsymbol{x}_{\pm \underline{\omega}}}}{\left| \boldsymbol{y}-\boldsymbol{x}_{\pm \underline{\omega}} \right|^{p+2}}.
	\end{align*}
	Hence we get 
	\begin{align*}
		&\lim_{\epsilon \rightarrow 0} \int_{\partial B_{\epsilon}}{\left( \alpha \frac{x_{p_i}-y_{p_i}}{\left| \boldsymbol{y}-\boldsymbol{x}_{\underline{\omega }} \right|^{p+2}}+\beta \frac{x_{p_i}-y_{p_i}}{\left| \boldsymbol{y}-\boldsymbol{x}_{-\underline{\omega }} \right|^{p+2}} \right) \overline{d\boldsymbol{y}^*}f\left( \boldsymbol{y} \right)}\\
		=&\sigma _{p+1}e_i\left[ \alpha f\left( \boldsymbol{x}_{\underline{\omega }} \right) +\beta f\left( \boldsymbol{x}_{-\underline{\omega }} \right) \right] .
	\end{align*}
	These give us that 
	\begin{align*}
		&\sigma _{p+1}\partial _{x_{p_i}}T_{\Omega _{\underline{\omega }}}f\left( \boldsymbol{x} \right)\\
		=&\int_{\Omega _{\underline{\omega }}}{\left[ \partial _{x_{p_i}}\left( \alpha \frac{\overline{\boldsymbol{y}-\boldsymbol{x}_{\underline{\omega }}}}{\left| \boldsymbol{y}-\boldsymbol{x}_{\underline{\omega }} \right|^{p+2}}+\beta \frac{\overline{\boldsymbol{y}-\boldsymbol{x}_{-\underline{\omega }}}}{\left| \boldsymbol{y}-\boldsymbol{x}_{-\underline{\omega }} \right|^{p+2}} \right) \right] f\left( \boldsymbol{y} \right) d\sigma _{\underline{\omega }}\left( \boldsymbol{y} \right)}\\
		&-\sigma _{p+1}e_i\left[ \alpha f\left( \boldsymbol{x}_{\underline{\omega }} \right) +\beta f\left( \boldsymbol{x}_{-\underline{\omega }} \right) \right] \\
		=&\sigma _{p+1}\int_{\Omega _{\underline{\omega }}}{\left( \partial _{x_{p_i}}\mathcal{E} _{\boldsymbol{y}}\left( \boldsymbol{x} \right) \right) f\left( \boldsymbol{y} \right) d\sigma _{\underline{\omega }}\left( \boldsymbol{y} \right)}-\sigma _{p+1}e_i\left[ \alpha f\left( \boldsymbol{x}_{\underline{\omega }} \right) +\beta f\left( \boldsymbol{x}_{-\underline{\omega }} \right) \right] ,
	\end{align*}
	which leads to
	\begin{align*}
		\partial _{x_{p_i}}T_{\Omega _{\underline{\omega }}}f\left( \boldsymbol{x} \right) =\int_{\Omega _{\underline{\omega }}}{\left( \partial _{x_{p_i}}\mathcal{E} _{\boldsymbol{y}}\left( \boldsymbol{x} \right) \right) f\left( \boldsymbol{y} \right) d\sigma _{\underline{\omega }}\left( \boldsymbol{y} \right)}-e_i\left[ \alpha f\left( \boldsymbol{x}_{\underline{\omega }} \right) +\beta f\left( \boldsymbol{x}_{-\underline{\omega }} \right) \right] .
	\end{align*}
	Therefore, we get 
	\begin{align*}
		&D_{\boldsymbol{x}_p}T_{\Omega _{\underline{\omega }}}f\left( \boldsymbol{x} \right) =\sum_{i=0}^p{e_i\partial _{x_{p_i}}T_{\Omega _{\underline{\omega }}}f\left( \boldsymbol{x} \right)}\\
		=&\sum_{i=0}^p{e_i\left[ \int_{\Omega _{\underline{\omega }}}{\left( \partial _{x_{p_i}}\mathcal{E} _{\boldsymbol{y}}\left( \boldsymbol{x} \right) \right) f\left( \boldsymbol{y} \right) d\sigma _{\underline{\omega }}\left( \boldsymbol{y} \right)}-e_i\left[ \alpha f\left( \boldsymbol{x}_{\underline{\omega }} \right) +\beta f\left( \boldsymbol{x}_{-\underline{\omega }} \right) \right] \right]}\\
		=&\int_{\Omega _{\underline{\omega }}}{\left( D_{\boldsymbol{x}_p}\mathcal{E} _{\boldsymbol{y}}\left( \boldsymbol{x} \right) \right) f\left( \boldsymbol{y} \right) d\sigma _{\underline{\omega }}\left( \boldsymbol{y} \right)}+\left[ \alpha f\left( \boldsymbol{x}_{\underline{\omega }} \right) +\beta f\left( \boldsymbol{x}_{-\underline{\omega }} \right) \right].
	\end{align*}
	Next, we consider $\partial _{x_{q_i}}T_{\Omega _{\underline{\omega }}}f\left( \boldsymbol{x} \right) $ and we notice that 
	\begin{align*}
		\partial _{x_{q_i}}=\frac{\partial r}{\partial x_{q_i}}\cdot \frac{\partial}{\partial r}=\frac{x_{q_i}}{\left| \underline{\boldsymbol{x}}_q \right|}\cdot \frac{\partial}{\partial r}.
	\end{align*}
	Then, with a similar argument as applied to $\partial _{x_{p_i}}T_{\Omega _{\underline{\omega }}}f\left( \boldsymbol{x} \right) $, we get
	\begin{align*}
		\partial _{x_{q_i}}\frac{1}{\left| \boldsymbol{y}-\boldsymbol{x}_{\pm \underline{\omega }} \right|^p}=\frac{x_{q_i}}{\left| \underline{\boldsymbol{x}}_q \right|}\cdot \frac{\partial}{\partial r}\frac{1}{\left| \boldsymbol{y}-\boldsymbol{x}_{\pm \underline{\omega }} \right|^p}=\frac{x_{q_i}}{\left| \underline{\boldsymbol{x}}_q \right|}\cdot \left( -p \right) \frac{r\mp \tilde{r}}{\left| \boldsymbol{y}-\boldsymbol{x}_{\pm \underline{\omega }} \right|^{p+2}}.
	\end{align*}
	Hence, we get
	\begin{align}\label{partial xqi}
		&\sigma _{p+1}\partial _{x_{q_i}}T_{\Omega _{\underline{\omega }}}f\left( \boldsymbol{x} \right) \nonumber \\ 
		=&\frac{1}{ -p}\partial _{x_{q_i}}\int_{\Omega _{\underline{\omega }}}{\left( \alpha \frac{1}{\left| \boldsymbol{y}-\boldsymbol{x}_{\underline{\omega }} \right|^p}+\frac{1}{\left| \boldsymbol{y}-\boldsymbol{x}_{-\underline{\omega }} \right|^p} \right) \left( D_{\boldsymbol{y}_{\underline{\omega }}}f\left( \boldsymbol{y} \right) \right) d\sigma _{\underline{\omega }}\left( \boldsymbol{y} \right)} \nonumber \\
		&-\frac{1}{-p}\partial _{x_{q_i}}\int_{\partial \Omega _{\underline{\omega }}}{\left( \alpha \frac{1}{\left| \boldsymbol{y}-\boldsymbol{x}_{\underline{\omega }} \right|^p}+\frac{1}{\left| \boldsymbol{y}-\boldsymbol{x}_{-\underline{\omega }} \right|^p} \right) \overline{d\boldsymbol{y}^*}f\left( \boldsymbol{y} \right)} \nonumber\\
		=&\frac{x_{q_i}}{\left| \underline{\boldsymbol{x}}_q \right|}\int_{\Omega _{\underline{\omega }}}{\left( \alpha \frac{r-\tilde{r}}{\left| \boldsymbol{y}-\boldsymbol{x}_{\underline{\omega }} \right|^{p+2}}+\beta \frac{r+\tilde{r}}{\left| \boldsymbol{y}-\boldsymbol{x}_{-\underline{\omega }} \right|^{p+2}} \right) \left( D_{\boldsymbol{y}_{\underline{\omega }}}f\left( \boldsymbol{y} \right) \right) d\sigma _{\underline{\omega }}\left( \boldsymbol{y} \right)} \nonumber\\
		&-\frac{x_{q_i}}{\left| \underline{\boldsymbol{x}}_q \right|}\int_{\partial \Omega _{\underline{\omega }}}{\left( \alpha \frac{r-\tilde{r}}{\left| \boldsymbol{y}-\boldsymbol{x}_{\underline{\omega }} \right|^{p+2}}+\beta \frac{r+\tilde{r}}{\left| \boldsymbol{y}-\boldsymbol{x}_{-\underline{\omega }} \right|^{p+2}} \right) \overline{d\boldsymbol{y}^*}f\left( \boldsymbol{y} \right)}.
	\end{align}
	Further, Gauss theorem tells us that 
	\begin{align*}
		&\left( \int_{\partial \Omega _{\underline{\omega}}}{-\int_{\partial B_{\epsilon}}} \right) \left( \alpha \frac{r-\tilde{r}}{\left| \boldsymbol{y}-\boldsymbol{x}_{\underline{\omega }} \right|^{p+2}}+\beta \frac{r+\tilde{r}}{\left| \boldsymbol{y}-\boldsymbol{x}_{-\underline{\omega }} \right|^{p+2}} \right)\overline{d\boldsymbol{y}^*}f\left( \boldsymbol{y} \right)\\
		=&\int_{\partial \Omega _{\underline{\omega}}\backslash B_{\epsilon}}{\left[ \left( \alpha \frac{r-\tilde{r}}{\left| \boldsymbol{y}-\boldsymbol{x}_{\underline{\omega }} \right|^{p+2}}+\beta \frac{r+\tilde{r}}{\left| \boldsymbol{y}-\boldsymbol{x}_{-\underline{\omega }} \right|^{p+2}} \right) D_{\boldsymbol{y}_{\underline{\omega}}} \right] f\left( \boldsymbol{y} \right)}\\
		& +\left( \alpha \frac{r-\tilde{r}}{\left| \boldsymbol{y}-\boldsymbol{x}_{\underline{\omega }} \right|^{p+2}}+\beta \frac{r+\tilde{r}}{\left| \boldsymbol{y}-\boldsymbol{x}_{-\underline{\omega }} \right|^{p+2}} \right) \left( D_{\boldsymbol{y}_{\underline{\omega}}}f\left( \boldsymbol{y} \right) \right) d\sigma _{\underline{\omega}}\left( \boldsymbol{y} \right) .
	\end{align*}
	Applying Gauss theorem, we have
	\begin{align*}
		&\sigma _{p+1}\partial _{x_{q_i}}T_{\Omega _{\underline{\omega }}}f\left( \boldsymbol{x} \right) \\
		=&\frac{x_{q_i}}{\left| \underline{\boldsymbol{x}}_q \right|}\lim_{\epsilon \rightarrow 0} \int_{B_{\epsilon}}{\left( \alpha \frac{r-\tilde{r}}{\left| \boldsymbol{y}-\boldsymbol{x}_{\underline{\omega }} \right|^{p+2}}+\beta \frac{r+\tilde{r}}{\left| \boldsymbol{y}-\boldsymbol{x}_{-\underline{\omega }} \right|^{p+2}} \right) \left( D_{\boldsymbol{y}_{\underline{\omega }}}f\left( \boldsymbol{y} \right) \right) d\sigma _{\underline{\omega }}\left( \boldsymbol{y} \right)}\\
		&-\frac{x_{q_i}}{\left| \underline{\boldsymbol{x}}_q \right|}\int_{\Omega _{\underline{\omega }}\backslash B_{\epsilon}}{\left[ \left( \alpha \frac{r-\tilde{r}}{\left| \boldsymbol{y}-\boldsymbol{x}_{\underline{\omega }} \right|^{p+2}}+\beta \frac{r+\tilde{r}}{\left| \boldsymbol{y}-\boldsymbol{x}_{-\underline{\omega }} \right|^{p+2}} \right) D_{\boldsymbol{y}_{\underline{\omega }}} \right] f\left( \boldsymbol{y} \right) d\sigma _{\underline{\omega }}\left( \boldsymbol{y} \right)}\\
		&-\frac{x_{q_i}}{\left| \underline{\boldsymbol{x}}_q \right|}\int_{\partial B_{\epsilon}}{\left( \alpha \frac{r-\tilde{r}}{\left| \boldsymbol{y}-\boldsymbol{x}_{\underline{\omega }} \right|^{p+2}}+\beta \frac{r+\tilde{r}}{\left| \boldsymbol{y}-\boldsymbol{x}_{-\underline{\omega }} \right|^{p+2}} \right) \overline{d\boldsymbol{y}^*}f\left( \boldsymbol{y} \right)}.
	\end{align*}
	Now, we notice that 
	\begin{align*}
		\frac{r\mp \tilde{r}}{\left| \boldsymbol{y}-\boldsymbol{x}_{\pm \underline{\omega }} \right|^{p+2}}=\frac{1}{-p}\frac{1}{\left| \boldsymbol{y}-\boldsymbol{x}_{\pm \underline{\omega }} \right|^p}\frac{\partial}{\partial r},
	\end{align*}
	which leads to 
	\begin{align*}
		&\frac{r\mp \tilde{r}}{\left| \boldsymbol{y}-\boldsymbol{x}_{\pm \underline{\omega }} \right|^{p+2}}D_{\boldsymbol{y}_{\underline{\omega }}}=\frac{1}{-p}\frac{1}{\left| \boldsymbol{y}-\boldsymbol{x}_{\pm \underline{\omega }} \right|^p}\frac{\partial}{\partial r}D_{\boldsymbol{y}_{\underline{\omega }}}\\
		=&\frac{1}{-p}\frac{1}{\left| \boldsymbol{y}-\boldsymbol{x}_{\pm \underline{\omega }} \right|^p}D_{\boldsymbol{y}_{\underline{\omega }}}\frac{\partial}{\partial r}=\frac{\overline{\boldsymbol{y}-\boldsymbol{x}_{\pm \underline{\omega }}}}{\left| \boldsymbol{y}-\boldsymbol{x}_{\pm \underline{\omega }} \right|^{p+2}}\frac{\partial}{\partial r}.
	\end{align*}
	Therefore, we can obtain 
	\begin{align*}
		&\sigma _{p+1}\partial _{x_{q_i}}T_{\Omega _{\underline{\omega }}}f\left( \boldsymbol{x} \right)\\
		&=\frac{x_{q_i}}{\left| \underline{\boldsymbol{x}}_q \right|}\lim_{\epsilon \rightarrow 0} \int_{\Omega _{\underline{\omega }}}{\frac{\partial}{\partial r}\left( \alpha \frac{\overline{\boldsymbol{y}-\boldsymbol{x}_{\underline{\omega }}}}{\left| \boldsymbol{y}-\boldsymbol{x}_{\underline{\omega }} \right|^{p+2}}+\beta \frac{\overline{\boldsymbol{y}-\boldsymbol{x}_{-\underline{\omega }}}}{\left| \boldsymbol{y}-\boldsymbol{x}_{-\underline{\omega }} \right|^{p+2}} \right) f\left( \boldsymbol{y} \right) d\sigma _{\underline{\omega }}\left( \boldsymbol{y} \right)}\\
		&\quad -\frac{x_{q_i}}{\left| \underline{\boldsymbol{x}}_q \right|}\int_{\partial B_{\epsilon}}{\left( \alpha \frac{r-\tilde{r}}{\left| \boldsymbol{y}-\boldsymbol{x}_{\underline{\omega }} \right|^{p+2}}+\beta \frac{r+\tilde{r}}{\left| \boldsymbol{y}-\boldsymbol{x}_{-\underline{\omega }} \right|^{p+2}} \right) \overline{d\boldsymbol{y}^*}f\left( \boldsymbol{y} \right)}\\
		&=\frac{x_{q_i}}{\left| \underline{\boldsymbol{x}}_q \right|}\sigma _{p+1}\left[ \int_{\Omega _{\underline{\omega }}}{\left( \frac{\partial}{\partial r}\mathcal{E} _{\boldsymbol{y}}\left( \boldsymbol{x} \right) \right) f\left( \boldsymbol{y} \right) d\sigma _{\underline{\omega }}\left( \boldsymbol{y} \right)}-\left( \alpha \underline{\omega }f\left( \boldsymbol{x}_{\underline{\omega }} \right) -\beta \underline{\omega }f\left( \boldsymbol{x}_{-\underline{\omega }} \right) \right) \right] ,
	\end{align*}
	which gives us 
	\begin{align*}
		&\partial _{x_{q_i}}T_{\Omega _{\underline{\omega }}}f\left( \boldsymbol{x} \right)\\ =&\frac{x_{q_i}}{\left| \underline{\boldsymbol{x}}_q \right|}\left[ \int_{\Omega _{\underline{\omega }}}{\left( \frac{\partial}{\partial r}\mathcal{E} _{\boldsymbol{y}}\left( \boldsymbol{x} \right) \right) f\left( \boldsymbol{y} \right) d\sigma _{\underline{\omega }}\left( \boldsymbol{y} \right)}-\left( \alpha \underline{\omega }f\left( \boldsymbol{x}_{\underline{\omega }} \right) -\beta \underline{\omega }f\left( \boldsymbol{x}_{-\underline{\omega }} \right) \right) \right] .
	\end{align*}
	Further, since
	\begin{align*}
		\mathbb{E} _{\underline{\boldsymbol{x}}_q}=\sum_{i=p+1}^{p+q}{x_i\partial _{x_{q_i}}}=\sum_{i=p+1}^{p+q}{x_i\frac{\partial r}{\partial x_{q_i}}\cdot}\frac{\partial}{\partial r}=\left| \underline{\boldsymbol{x}}_q \right|\frac{\partial}{\partial r},
	\end{align*}
	we get
	\begin{align*}
		&\sigma _{p+1}\mathbb{E} _{\underline{\boldsymbol{x}}_q}T_{\Omega _{\underline{\omega }}}f\left( \boldsymbol{x} \right)\\
		=&\left| \underline{\boldsymbol{x}}_q \right|\sigma _{p+1}\bigg[ \int_{\Omega _{\underline{\omega }}}{\left( \frac{\partial}{\partial r}\mathcal{E} _{\boldsymbol{y}}\left( \boldsymbol{x} \right) \right) f\left( \boldsymbol{y} \right) d\sigma _{\underline{\omega }}\left( \boldsymbol{y} \right)}
		-\left( \alpha \underline{\omega }f\left( \boldsymbol{x}_{\underline{\omega }} \right) -\beta \underline{\omega }f\left( \boldsymbol{x}_{-\underline{\omega }} \right) \right) \bigg] \\
		=&\frac{\left| \underline{\boldsymbol{x}}_q \right|^2 \sigma _{p+1}}{\underline{\boldsymbol{x}}_q}\bigg[ \int_{\Omega _{\underline{\omega }}}{\underline{\eta }\bigg(\frac{\partial}{\partial r}\mathcal{E} _{\boldsymbol{y}}\left( \boldsymbol{x} \right) \bigg) f\left( \boldsymbol{y} \right) d\sigma _{\underline{\omega }}\left( \boldsymbol{y} \right)}\\
		&-\underline{\eta }\left( \alpha \underline{\omega }f\left( \boldsymbol{x}_{\underline{\omega }} \right) -\beta \underline{\omega }f\left( \boldsymbol{x}_{-\underline{\omega }} \right) \right) \bigg]  \\
		=&\frac{\left| \underline{\boldsymbol{x}}_q \right|^2\sigma _{p+1}}{\underline{\boldsymbol{x}}_q}\left[ \int_{\Omega _{\underline{\omega }}}{\underline{\eta }\left( \frac{\partial}{\partial r}\mathcal{E} _{\boldsymbol{y}}\left( \boldsymbol{x} \right) \right) f\left( \boldsymbol{y} \right) d\sigma _{\underline{\omega }}\left( \boldsymbol{y} \right)}+\left( \alpha f\left( \boldsymbol{x}_{\underline{\omega }} \right) +\beta f\left( \boldsymbol{x}_{-\underline{\omega }} \right) \right) \right] \\
		=&\frac{\left| \underline{\boldsymbol{x}}_q \right|^2}{\underline{\boldsymbol{x}}_q}\cdot \sigma _{p+1}\left[ \int_{\Omega _{\underline{\omega }}}{\underline{\eta }\left( \frac{\partial}{\partial r}\mathcal{E} _{\boldsymbol{y}}\left( \boldsymbol{x} \right) \right) f\left( \boldsymbol{y} \right) d\sigma _{\underline{\omega }}\left( \boldsymbol{y} \right)}+f\left( \boldsymbol{x} \right) \right] ,
	\end{align*}
	which shows us that
	\begin{align*}
		\mathbb{E} _{\underline{\boldsymbol{x}}_q}T_{\Omega _{\underline{\omega }}}f\left( \boldsymbol{x} \right) =\frac{\left| \underline{\boldsymbol{x}}_q \right|^2}{\underline{\boldsymbol{x}}_q}\left[ \int_{\Omega _{\underline{\omega }}}{\underline{\eta }\left( \frac{\partial}{\partial r}\mathcal{E} _{\boldsymbol{y}}\left( \boldsymbol{x} \right) \right) f\left( \boldsymbol{y} \right) d\sigma _{\underline{\omega }}\left( \boldsymbol{y} \right)}+f\left( \boldsymbol{x} \right) \right] .
	\end{align*}
	According to the definition of $\bar{\vartheta}$ , we have 
	\begin{align*}
		\bar{\vartheta}T_{\Omega _{\underline{\omega}}}f\left( \boldsymbol{x} \right) &=D_{\boldsymbol{x}_p}T_{\Omega _{\underline{\omega }}}f\left( \boldsymbol{x} \right) +\frac{\underline{\boldsymbol{x}}_q}{\left| \underline{\boldsymbol{x}}_q \right|^2}\mathbb{E} _{\underline{\boldsymbol{x}}_q}T_{\Omega _{\underline{\omega }}}f\left( \boldsymbol{x} \right) \\
		&=\int_{\Omega _{\underline{\omega}}}{\left( D_{\boldsymbol{x}_p}+\underline{\eta}\frac{\partial}{\partial r} \right) f\left( \boldsymbol{y} \right) d\sigma _{\underline{\omega}}\left( \boldsymbol{y} \right)}+2f\left( \boldsymbol{x} \right) \\
		&=\int_{\Omega _{\underline{\omega}}}{\bar{\vartheta}f\left( \boldsymbol{y} \right) d\sigma _{\underline{\omega}}\left( \boldsymbol{y} \right)}+2f\left( \boldsymbol{x} \right) 
		=2f\left( \boldsymbol{x} \right),
	\end{align*}
	which completes the proof.
\end{proof}
\begin{remark}
	The operator $\bar{\vartheta}$ restricted to $\Omega _{\underline{\eta }}$ is actually equal to the operator $D_{\underline{\eta}}$ mentioned in \cite{Xu1}.
\end{remark}
\begin{lemma}\label{lemma xpiT}
	Let $\Omega _D\subset \mathbb{R} _{*}^{p+q+1}$ be a bounded and $p$-symmetric domain, $f\in C^1\left( \overline{\Omega _D} \right) \cap \mathcal{G} \mathcal{S} \left( \Omega _D \right)$, then we have
	\begin{align*}
		\partial _{x_{q_i}}T_{\Omega _D}f\left( \boldsymbol{x} \right) =\int_{\mathbb{S} ^+}{\partial _{x_{q_i}}T_{\Omega _{\underline{\omega}}}f\left( \boldsymbol{x} \right) dS_{\underline{\omega}} \left( \boldsymbol{y}\right)}.
	\end{align*}
\end{lemma}
\begin{proof}
	Let $\boldsymbol{y} = \boldsymbol{y}_p + \boldsymbol{y}_q \in \Omega_D$. By representing $\boldsymbol{y}_q$ in spherical coordinates as $\boldsymbol{y}_q = r\underline{\omega}$ with $\underline{\omega} \in \mathbb{S}$, we can express the volume element $d\sigma (\boldsymbol{y})$ as  $d\sigma (\boldsymbol{y}) = r^{q-1} \, d\sigma _{\underline{\omega}}(\boldsymbol{y}) \, dS_{\underline{\omega}} \left( \boldsymbol{y}\right)$. Consequently,
	\begin{align*}
		T_{\Omega _D}f\left( \boldsymbol{x} \right) &=-\int_{\Omega _D}{K_{\boldsymbol{y}}\left( \boldsymbol{x} \right) f\left( \boldsymbol{y} \right) d\sigma \left( \boldsymbol{y} \right)}\\
		&=-\int_{\mathbb{S} ^+}{\int_{\Omega _{\underline{\omega }}}{\frac{\mathcal{E} _{\boldsymbol{y}}\left( \boldsymbol{x} \right)}{\sigma _{q-1}|\underline{\boldsymbol{y}}_q|^{q-1}}\cdot f\left( \boldsymbol{y} \right) r^{q-1}d\sigma _{\underline{\omega }}\left( \boldsymbol{y} \right)}dS_{\underline{\omega}} \left( \boldsymbol{y}\right)}\\
		&=-\frac{1}{\sigma _{q-1}}\int_{\mathbb{S} ^+}{\int_{\Omega _{\underline{\omega }}}{\mathcal{E} _{\boldsymbol{y}}\left( \boldsymbol{x} \right) f\left( \boldsymbol{y} \right) d\sigma _{\underline{\omega }}\left( \boldsymbol{y} \right)}dS_{\underline{\omega}} \left( \boldsymbol{y}\right)}\\
		&=\frac{1}{\sigma _{q-1}}\int_{\mathbb{S} ^+}{T_{\Omega _{\underline{\omega }}}f\left( \boldsymbol{x} \right) dS_{\underline{\omega}} \left( \boldsymbol{y}\right)}.
	\end{align*}
	Now, we proceed to validate the two prerequisites stated in Theorem \ref{interchange}. Initially, the existence of $\partial _{x_{q_i}}T_{\Omega _{\underline{\omega }}}f\left( \boldsymbol{x} \right)$ has been previously substantiated in the preceding theorem. Subsequently, based on the reasoning of the aforementioned theorem, coupled with equations $(\ref{partial xpi})$ and $(\ref{partial xqi})$, and fact that $f\in \mathcal{G} \mathcal{S} \left( \Omega _D \right)$, we infer that
	\begin{align*}
		&\sigma _{p+1}\partial _{x_{p_i}}T_{\Omega _{\underline{\omega }}}f\left( \boldsymbol{x} \right)\\
		=&\int_{\Omega _{\underline{\omega }}}{\left( \alpha \frac{x_{p_i}-y_{p_i}}{\left| \boldsymbol{y}-\boldsymbol{x}_{\underline{\omega }} \right|^{p+2}}+\beta \frac{x_{p_i}-y_{p_i}}{\left| \boldsymbol{y}-\boldsymbol{x}_{-\underline{\omega }} \right|^{p+2}} \right) \left( D_{\boldsymbol{y}_{\underline{\omega }}}f\left( \boldsymbol{y} \right) \right) d\sigma _{\underline{\omega }}\left( \boldsymbol{y} \right)}\\
		&-\int_{\partial \Omega _{\underline{\omega }}}{\left( \alpha \frac{x_{p_i}-y_{p_i}}{\left| \boldsymbol{y}-\boldsymbol{x}_{\underline{\omega }} \right|^{p+2}}+\beta \frac{x_{p_i}-y_{p_i}}{\left| \boldsymbol{y}-\boldsymbol{x}_{-\underline{\omega }} \right|^{p+2}} \right) \overline{d\boldsymbol{y}^*}f\left( \boldsymbol{y} \right)},\\
		&\sigma _{p+1}\partial _{x_{q_i}}T_{\Omega _{\underline{\omega }}}f\left( \boldsymbol{x} \right) \\ 
		=&\frac{x_{q_i}}{\left| \underline{\boldsymbol{x}}_q \right|}\int_{\Omega _{\underline{\omega }}}{\left( \alpha \frac{r-\tilde{r}}{\left| \boldsymbol{y}-\boldsymbol{x}_{\underline{\omega }} \right|^{p+2}}+\beta \frac{r+\tilde{r}}{\left| \boldsymbol{y}-\boldsymbol{x}_{-\underline{\omega }} \right|^{p+2}} \right) \left( D_{\boldsymbol{y}_{\underline{\omega }}}f\left( \boldsymbol{y} \right) \right) d\sigma _{\underline{\omega }}\left( \boldsymbol{y} \right)} \\
		&-\frac{x_{q_i}}{\left| \underline{\boldsymbol{x}}_q \right|}\int_{\partial \Omega _{\underline{\omega }}}{\left( \alpha \frac{r-\tilde{r}}{\left| \boldsymbol{y}-\boldsymbol{x}_{\underline{\omega }} \right|^{p+2}}+\beta \frac{r+\tilde{r}}{\left| \boldsymbol{y}-\boldsymbol{x}_{-\underline{\omega }} \right|^{p+2}} \right) \overline{d\boldsymbol{y}^*}f\left( \boldsymbol{y} \right)}.
	\end{align*}
	Given the homogeneity of  $\frac{x_{p_i}-y_{p_i}}{|\boldsymbol{y}-\boldsymbol{x}_{\pm \underline{\omega }}|^{p+2}}$ and $\frac{r\mp \tilde{r}}{|\boldsymbol{y}-\boldsymbol{x}_{\pm \underline{\omega }}|^{p+2}}$ and $f\in C^1\left( \overline{\Omega _D} \right) $, it becomes evident that $\partial _{x_{p_i}}T_{\Omega _{\underline{\omega }}}f\left( \boldsymbol{x} \right)$, $\partial _{x_{q_i}}T_{\Omega _{\underline{\omega }}}f\left( \boldsymbol{x} \right)$ are integrable over $\mathbb{S}$ for all $\boldsymbol{x}\in \Omega _D$. Indeed, it is evident that 
	\begin{align*}
		|\partial _{x_{p_i}}T_{\Omega _{\underline{\omega }}}f\left( \boldsymbol{x} \right) |\leqslant &C_1\int_{\Omega _{\underline{\omega }}}{\frac{|x_{p_i}-y_{p_i}|}{|\boldsymbol{y}-\boldsymbol{x}_{\underline{\omega }}|^{p+2}}|D_{\boldsymbol{y}_{\underline{\omega }}}f\left( \boldsymbol{y} \right) |d\sigma _{\underline{\omega }}\left( \boldsymbol{y} \right)}\\
		&+C_2\int_{\Omega _{\underline{\omega }}}{\frac{|x_{p_i}-y_{p_i}|}{|\boldsymbol{y}-\boldsymbol{x}_{-\underline{\omega }}|^{p+2}}|D_{\boldsymbol{y}_{\underline{\omega }}}f\left( \boldsymbol{y} \right) |d\sigma _{\underline{\omega }}\left( \boldsymbol{y} \right)}+C^{\prime}.
	\end{align*}
	Since the integrals above are well-defined as Cauchy's principal values, we consider, for example, the limit of the integral  
	\begin{align*}
		\int_{\Omega _{\underline{\omega }}\backslash B_{\epsilon}( \boldsymbol{x}_{\underline{\omega}} )}{\frac{|x_{p_i}-y_{p_i}|}{|\boldsymbol{y}-\boldsymbol{x}_{\underline{\omega }}|^{p+2}}|D_{\boldsymbol{y}_{\underline{\omega }}}f\left( \boldsymbol{y} \right) |d\sigma _{\underline{\omega}}\left( \boldsymbol{y} \right)}.
	\end{align*}
	When using spherical coordinates, we observe that $d\sigma (\boldsymbol{y})$ incorporates $|\boldsymbol{y}-\boldsymbol{x}_{\underline{\omega }}|$ as a factor. Additionally, given $f \in C^1(\varOmega_D)$ and that $\Omega_D$ is bounded and closed, we can select $R > 0$ such that $\Omega _D\subset B_R\left( \boldsymbol{x}_{\underline{\omega}} \right)$ and we assume that $f\left( \boldsymbol{y} \right) =0$. Furthermore, we assume $f(\boldsymbol{y}) = 0$ for $\boldsymbol{y}\in B_R\left( \boldsymbol{x}_{\underline{\omega}} \right) \backslash \Omega _D$.  Then, we set $\boldsymbol{y}=\boldsymbol{x}_{\underline{\omega }}+s\boldsymbol{t}$, where $\boldsymbol{t}=\boldsymbol{t}_p+\underline{\omega }t_1\in \Omega _{\underline{\omega }}$ and we get
	\begin{align*}
		&\int_{\mathbb{S} ^+}{\int_{\Omega _{\underline{\omega }}\backslash B_{\epsilon}\left( \boldsymbol{x}_{\underline{\omega }} \right)}{\frac{|x_{p_i}-y_{p_i}|}{|\boldsymbol{y}-\boldsymbol{x}_{\underline{\omega }}|^{p+2}}|D_{\boldsymbol{y}_{\underline{\omega }}}f\left( \boldsymbol{y} \right) |d\sigma _{\underline{\omega }}\left( \boldsymbol{y} \right)}}dS_{\underline{\omega}} \left( \boldsymbol{y}\right)\\
		=&\int_{\epsilon}^R{\int_{|\boldsymbol{t}|=1}{\frac{|st_{p_i}|}{|-s\boldsymbol{t}|^{p+2}}\cdot s^{p+1}|D_{\boldsymbol{t}_{\underline{\omega }}}f\left( \boldsymbol{x}_{\underline{\omega }}+s\boldsymbol{t} \right) |}d\sigma _{\underline{\omega }}\left( \boldsymbol{t} \right)}dS_{\underline{\omega}} \left( \boldsymbol{t}\right)\\
		\leqslant &C\int_{\epsilon}^R{\int_{|\boldsymbol{t}|=1}{|D_{\boldsymbol{t}_{\underline{\omega }}}f\left( \boldsymbol{x}_{\underline{\omega }}+s\boldsymbol{t} \right) |}d\sigma _{\underline{\omega }}\left( \boldsymbol{t} \right)}dS_{\underline{\omega}} \left( \boldsymbol{t}\right)\leqslant C.
	\end{align*}
	Since $f \in C^1(\overline{\Omega_D})$, the integral is finite, verifying the second condition of Theorem \ref{interchange}, thus completing the proof. An analogous argument holds for $\partial _{x_{q_i}}T_{\Omega _{\underline{\omega }}}f\left( \boldsymbol{x} \right)$.
\end{proof}
\begin{theorem}
	Let $\Omega _D\subset \mathbb{R} _{*}^{p+q+1}$ be a bounded and $p$-symmetric domain, $f\in C^1\left( \overline{\Omega _D} \right) \cap \mathcal{GS} \left( \Omega _D \right)$, then we have
	\begin{align*}
		\bar{\vartheta}_{\boldsymbol{x}}T_{\Omega _D}f\left( \boldsymbol{x} \right) =f\left( \boldsymbol{x} \right).
	\end{align*}
\end{theorem}
\begin{proof}
	Let $\boldsymbol{y} = \boldsymbol{y}_p + \boldsymbol{y}_q \in \Omega_D$. We rewrite $\boldsymbol{y}_q = r\underline{\omega}$ with $\underline{\omega} \in \mathbb{S}$ using spherical coordinates. Then, the volume element can be expressed as $d\sigma (\boldsymbol{y}) = r^{q-1} \, d\sigma _{\underline{\omega}}(\boldsymbol{y}) \, dS_{\underline{\omega}} \left( \boldsymbol{y}\right)$. Hence, we have
	\begin{align*}
		T_{\Omega _D}f\left( \boldsymbol{x} \right) &=-\int_{\Omega _D}{K_{\boldsymbol{y}}\left( \boldsymbol{x} \right) f\left( \boldsymbol{y} \right) d\sigma \left( \boldsymbol{y} \right)}\\
		&=-\int_{\mathbb{S} ^+}{\int_{\Omega _{\underline{\omega }}}{\frac{\mathcal{E} _{\boldsymbol{y}}\left( \boldsymbol{x} \right)}{\sigma _{q-1}|\underline{\boldsymbol{y}}_q|^{q-1}}\cdot f\left( \boldsymbol{y} \right) r^{q-1}d\sigma _{\underline{\omega }}\left( \boldsymbol{y} \right)}dS_{\underline{\omega}} \left( \boldsymbol{y}\right)}\\
		&=-\frac{1}{\sigma _{q-1}}\int_{\mathbb{S} ^+}{\int_{\Omega _{\underline{\omega }}}{\mathcal{E} _{\boldsymbol{y}}\left( \boldsymbol{x} \right) f\left( \boldsymbol{y} \right) d\sigma _{\underline{\omega }}\left( \boldsymbol{y} \right)}dS_{\underline{\omega}} \left( \boldsymbol{y}\right)}\\
		&=\frac{1}{\sigma _{q-1}}\int_{\mathbb{S} ^+}{T_{\Omega _{\underline{\omega }}}f\left( \boldsymbol{x} \right) dS_{\underline{\omega}} \left( \boldsymbol{y}\right)}.
	\end{align*}
	By Lemma \ref{lemma xpiT} and Theorem \ref{Theta T=2f}, we have
	\begin{align*}
		\bar{\vartheta}_{\boldsymbol{x}}T_{\Omega _D}f\left( \boldsymbol{x} \right) &=\frac{1}{\sigma _{q-1}}\int_{\mathbb{S} ^+}{\bar{\vartheta}_{\boldsymbol{x}}T_{\Omega _{\underline{\omega }}}f\left( \boldsymbol{x} \right) dS_{\underline{\omega}} \left( \boldsymbol{y}\right)}\\
		&=\frac{1}{\sigma _{q-1}}\int_{\mathbb{S} ^+}{2f\left( \boldsymbol{x} \right) dS_{\underline{\omega}} \left( \boldsymbol{y}\right)}=f\left( \boldsymbol{x} \right).
	\end{align*}
\end{proof}
\par
Given the proof of Lemma \ref{lemma xpiT}, for $f\in C^1\left( \overline{\Omega _D} \right) \cap \mathcal{GS} \left( \Omega _D \right)$, we immediately conclude that
\begin{align*}
	& \partial _{x_{p_i}}T_{\Omega _D}f\left( \boldsymbol{x} \right) \nonumber\\
	=& \int_{\mathbb{S} ^{^+}}{\bigg[ \int_{\Omega _{\underline{\omega }}}{\frac{-p}{\sigma _{p+1}}\left( \alpha \frac{x_{p_i}-y_{p_i}}{|\boldsymbol{y}-\boldsymbol{x}_{\underline{\omega }}|^{p+2}}+\beta \frac{x_{p_i}-y_{p_i}}{|\boldsymbol{y}-\boldsymbol{x}_{-\underline{\omega }}|^{p+2}} \right) f\left( \boldsymbol{y} \right) d\sigma _{\underline{\omega }}\left( \boldsymbol{y} \right)}}\\
	&+e_if\left( \boldsymbol{x} \right) \bigg]dS_{\underline{\omega}} \left( \boldsymbol{y}\right)\\
	=& \int_{\Omega _D}\frac{-p}{\sigma _{p+1}}\left( \alpha \frac{x_{p_i}-y_{p_i}}{\left| \boldsymbol{y}-\boldsymbol{x}_{\underline{\omega }} \right|^{p+2}}+\beta \frac{x_{p_i}-y_{p_i}}{\left| \boldsymbol{y}-\boldsymbol{x}_{-\underline{\omega }} \right|^{p+2}} \right) f\left( \boldsymbol{y} \right) \cdot \left| \underline{\boldsymbol{y}}_q \right|^{1-q}d\sigma \left( \boldsymbol{y} \right)\nonumber\\
	&+\frac{\sigma _{q-1}}{2}\cdot e_if\left( \boldsymbol{x} \right),\\
	&\partial _{x_{q_i}}T_{\Omega _D}f\left( \boldsymbol{x} \right)\\
	=&\frac{x_{q_i}}{\left| \underline{\boldsymbol{x}}_q \right|}\int_{\mathbb{S} ^+} \bigg[ {\int_{\Omega _{\underline{\omega }}}{\frac{-p}{\sigma _{p+1}}\left( \alpha \frac{r-\tilde{r}}{\left| \boldsymbol{y}-\boldsymbol{x}_{\underline{\omega }} \right|^{p+2}}+\beta \frac{r+\tilde{r}}{\left| \boldsymbol{y}-\boldsymbol{x}_{-\underline{\omega }} \right|^{p+2}} \right) f\left( \boldsymbol{y} \right) d\sigma _{\underline{\omega }}\left( \boldsymbol{y} \right)}}\\
	&-\left( \alpha \underline{\omega}f\left( \boldsymbol{x}_{\underline{\omega}} \right) -\beta \underline{\omega}f\left( \boldsymbol{x}_{-\underline{\omega}} \right) \right) \bigg] dS_{\underline{\omega}} \left( \boldsymbol{y}\right)\\
	=&\frac{x_{q_i}}{\left| \underline{\boldsymbol{x}}_q \right|}\int_{\Omega _D}{\frac{-p}{\sigma _{p+1}}\left( \alpha \frac{r-\tilde{r}}{\left| \boldsymbol{y}-\boldsymbol{x}_{\underline{\omega }} \right|^{p+2}}+\beta \frac{r+\tilde{r}}{\left| \boldsymbol{y}-\boldsymbol{x}_{-\underline{\omega }} \right|^{p+2}} \right) f\left( \boldsymbol{y} \right) \left| \underline{\boldsymbol{y}}_q \right|^{1-q}d\sigma \left( \boldsymbol{y} \right)}\\
	&-\frac{\sigma _{q-1}x_{q_i}\underline{\boldsymbol{x}}_q}{2\left| \underline{\boldsymbol{x}}_q \right|^2}f\left( \boldsymbol{x} \right).
\end{align*}
We define the $L^t$ space over an $p$-symmetric domain $\Omega _D$ for slice functions as 
\begin{align*}
	\mathcal{L} ^t\left( \Omega _D \right) =\mathcal{GS} \left( \Omega _D \right) \cap L^t\left( \Omega _D \right).
\end{align*}
\par

\section{Hodge decomposition on a Banach space}
Let $\Omega _D\subset \mathbb{R} _{*}^{p+q+1}$ be a bounded p-symmetric domain and $1<t<\infty$, the norm of a $\mathbb{R} _{p+q}-valued$ function $f\in L^t\left( \Omega _D \right)$ is given by 
\begin{align*}
	\left\| f \right\| _{L^t\left( \Omega _D \right)} := \left( \int_{\Omega _D}{\left| f\left( \boldsymbol{x} \right) \right|^td\sigma \left( \boldsymbol{x} \right)} \right) ^{\frac{1}{t}}.
\end{align*}
In addition, we define
\begin{align*}
	\mathcal{A} ^t\left( \Omega _D \right) :=ker\bar{\vartheta}\cap \mathcal{L} ^t\left( \Omega _D \right).
\end{align*}
\begin{proposition} \label{lambda K} \cite{Co7}
	Let $\Omega _D\subset \mathbb{R} _{*}^{p+q+1}$ be a bounded p-symmetric domain. For any compact set $K\subset \Omega _D$, there exists a constant $\lambda _K>0$ such that 
	\begin{align*}
		sup\left\{ \left| f\left( \boldsymbol{x} \right) \right|\,\,: \boldsymbol{x}\in K \right\} \leqslant \lambda _K\left\| f \right\| _{L^t}, \forall f\in \mathcal{A} ^t\left( \Omega _D \right) .    
	\end{align*}
\end{proposition}
\begin{proposition}
	Let $\Omega _D\subset \mathbb{R} _{*}^{p+q+1}$ be a bounded p-symmetric domain, then
	\begin{enumerate}
		\item $\mathcal{L} ^t\left( \Omega _D \right)$ is a closed subspace of $L^t\left( \Omega _D \right)$,
		\item $\mathcal{A} ^t\left( \Omega _D \right)$ is a closed subspace of $\mathcal{L} ^t\left( \Omega _D \right)$.
	\end{enumerate}
\end{proposition}
\begin{proof}
	\begin{enumerate}
		\item We suppose that there is a sequence  $f_n\in \mathcal{L} ^t\left( \Omega _D \right)$, which converges to $f \in L^t\left( \Omega _D \right)$ in the $L^t$ norm. If we want to prove $f\in \mathcal{GS} \left( \Omega _D \right)$, we just need to prove that $f$ satisfies the representation formula. For any $\underline{\omega}\in \mathbb{S}$, we get
		\begin{align}\label{fn representation formula}
			&\left\| f\left( \boldsymbol{x} \right) -\left( \alpha f\left( \boldsymbol{x}_{\underline{\omega}} \right) +\beta f\left( \boldsymbol{x}_{-\underline{\omega}} \right) \right) \right\| _{L^t}\nonumber\\
			\leqslant& \left\| f\left( \boldsymbol{x} \right) -f_n\left( \boldsymbol{x} \right) \right\| _{L^t}+\left\| f_n\left( \boldsymbol{x} \right) -\left( \alpha f_n\left( \boldsymbol{x}_{\underline{\omega}} \right) +\beta f_n\left( \boldsymbol{x}_{-\underline{\omega}} \right) \right) \right\| _{L^t}\nonumber\\
			&+\left\| \left( \alpha f_n\left( \boldsymbol{x}_{\underline{\omega}} \right) +\beta f_n\left( \boldsymbol{x}_{-\underline{\omega}} \right) \right) -\left( \alpha f\left( \boldsymbol{x}_{\underline{\omega}} \right) +\beta f\left( \boldsymbol{x}_{-\underline{\omega}} \right) \right) \right\| _{L^t}.
		\end{align}
		According to the assumption for $\left\{ f_n \right\}$, we know that 
		\begin{align*}
			\left\| f\left( \boldsymbol{x} \right) -f_n\left( \boldsymbol{x} \right) \right\| _{L^t}\longrightarrow 0 , \quad \quad \left( when\,\,n\rightarrow \infty \right) .
		\end{align*}
		Since $f_n$ is a generalized partial-slice function for all $n$, then we have $$f_n\left( \boldsymbol{x} \right) =\left( \alpha f_n\left( \boldsymbol{x}_{\underline{\omega }} \right) +\beta f_n\left( \boldsymbol{x}_{-\underline{\omega }} \right) \right).$$ Now, let's think about another formula as follows.
		\begin{align*}
			&\quad \left\| \left( \alpha f_n\left( \boldsymbol{x}_{\underline{\omega}} \right) +\beta f_n\left( \boldsymbol{x}_{-\underline{\omega}} \right) \right) -\left( \alpha f\left( \boldsymbol{x}_{\underline{\omega}} \right) +\beta f\left( \boldsymbol{x}_{-\underline{\omega}} \right) \right) \right\| _{L^t}\\
			&\leqslant \left\| \alpha \left( f_n\left( \boldsymbol{x}_{\underline{\omega }} \right) -f\left( \boldsymbol{x}_{\underline{\omega }} \right) \right) \right\| _{L^t}+\left\| \beta \left( f_n\left( \boldsymbol{x}_{-\underline{\omega }} \right) -f\left( \boldsymbol{x}_{-\underline{\omega }} \right) \right) \right\| _{L^t}.  
		\end{align*}
		Now, we assume $\boldsymbol{x}=\boldsymbol{x}_p+\underline{\eta}r$, since the sequence $\left\{ f_n \right\}$ converged to $f$ in the $L^t$ norm, which can be written as 
		\begin{eqnarray}\label{norm of (fn-f)}
			\begin{split}
				\left\| f_n-f \right\| _{L^t}^{t}&=\int_{\Omega _D}{\left| f_n\left( \boldsymbol{x} \right) -f\left( \boldsymbol{x} \right) \right|^td\sigma \left( \boldsymbol{x} \right)}\\
				&=\int_{\mathbb{S} ^+}{\int_{\Omega _{\underline{\eta}}}{\left| \left( f_n-f \right) \left( \boldsymbol{x}_p+\underline{\eta}r \right) \right|}^t\cdot r^{q-1}d\sigma _{\underline{\eta}}\left( \boldsymbol{x} \right) dS_{\underline{\eta}} \left( \boldsymbol{x}\right)}.
			\end{split}
		\end{eqnarray}
		Through our previous assumptions, equation $(\ref{norm of (fn-f)})$ converges to zero when $n$ goes to infinity. This implies that for almost every $\underline{\eta}\in \mathbb{S} ^+$, we get 
		\begin{align*}
			\int_{\Omega _{\underline{\eta}}}{\left| \left( f_n-f \right) \left( \boldsymbol{x} \right) \right|^t\cdot r^{q-1}d\sigma _{\underline{\eta}}\left( \boldsymbol{x} \right)}\longrightarrow 0,\quad \quad \left( when\,\,n\rightarrow \infty \right).
		\end{align*}
		Hence, for the constant $C$, $C^{\prime}$, we have 
		\begin{align*}
			&\left\| \alpha \left( f_n\left( \boldsymbol{x}_{\underline{\omega}} \right) -f_n\left( \boldsymbol{x}_{\underline{\omega}} \right) \right) \right\| _{L^t}^{t}\\
			= &\int_{\Omega _D}{\left| \alpha \left( f_n\left( \boldsymbol{x}_{\underline{\omega }} \right) -f_n\left( \boldsymbol{x}_{\underline{\omega }} \right) \right) \right|^td\sigma \left( \boldsymbol{x} \right)}\\
			\leqslant& C\int_{\mathbb{S} ^+}{\int_{\Omega _{\underline{\eta}}}{\left| \left( f_n-f \right) \left( \boldsymbol{x}_p+\underline{\omega}r \right) \right|^t\cdot r^{q-1}d\sigma _{\underline{\eta}}\left( \boldsymbol{x} \right)}dS_{\underline{\eta}} \left( \boldsymbol{x}\right)}\\
			=&C\int_{\mathbb{S} ^+}{\int_{\Omega _{\underline{\omega }}}{\left| \left( f_n-f \right) \left( \boldsymbol{x}_p+\underline{\omega }r \right) \right|^t\cdot r^{q-1}d\sigma _{\underline{\omega }}\left( \boldsymbol{x} \right)}dS_{\underline{\eta}} \left( \boldsymbol{x}\right)}\\
			\leqslant&C^{\prime}\int_{\Omega _{\underline{\eta}}}{\left| \left( f_n-f \right) \left( \boldsymbol{x}_p+\underline{\omega}r \right) \right|^t\cdot r^{q-1}d\sigma _{\underline{\eta}}\left( \boldsymbol{x} \right)}\longrightarrow 0,\quad \left( n\rightarrow \infty \right),
		\end{align*}
		where the last second equality comes from the fact that the domains of the variables $\boldsymbol{x}_p$, $r$ on $\Omega _{\underline{\omega}}$ and $\Omega _{\underline{\eta}}$ are the same. In a similar way, we come to similar conclusions
		\begin{align*}
			\left\| \beta \left( f_n\left( \boldsymbol{x}_{-\underline{\omega}} \right) -f\left( \boldsymbol{x}_{-\underline{\omega}} \right) \right) \right\| _{L^t}\longrightarrow 0,\quad \quad \left( n\rightarrow \infty \right) .
		\end{align*}
		Finally, according to equation $(\ref{fn representation formula})$, we have $$\left\| f\left( \boldsymbol{x} \right) -\left( \alpha f\left( \boldsymbol{x}_{\underline{\omega }} \right) +\beta f\left( \boldsymbol{x}_{-\underline{\omega }} \right) \right) \right\| _{L^t}\longrightarrow 0,$$ when $n\rightarrow \infty$, which completes the proof.
		\item Suppose $\left\{ f_n \right\}$ is a convergent sequence in $\mathcal{A} ^t\left( \Omega _D \right)$ converging to $f$, where $f$ is its limit function in $\mathcal{L} ^t\left( \Omega _D \right)$. By proposition $\ref{lambda K}$, we know that there exists a function $g\,:\, \Omega _D\longrightarrow \mathbb{R} _{p+q}$ given by
		\begin{align*}
			g\left( \boldsymbol{x} \right) \,\,:=\,\,\underset{n\rightarrow \infty}{\lim}f_n\left( \boldsymbol{x} \right) ,\quad for\,\,all\,\,\boldsymbol{x}\in \Omega _D,
		\end{align*}
		and $\left\{ f_n \right\}$ converges uniformly to $g$ on compact subsets of $\Omega _D$, which implies that $g$ is a generalized partial-slice monogenic function on $\Omega _D$. Now for any compact subset $K\subset \Omega _D$, we have
		\begin{align*}
			0&\leqslant \int_K{\left| \left( f-g \right) \left( \boldsymbol{x} \right) \right|^td\sigma \left( \boldsymbol{x} \right)}\\
			&\leqslant \int_K{\left| \left( f-f_n \right) \left( \boldsymbol{x} \right) \right|^td\sigma \left( \boldsymbol{x} \right)}+\int_K{\left| \left( f_n-g \right) \left( \boldsymbol{x} \right) \right|^td\sigma \left( \boldsymbol{x} \right)}\\
			&\leqslant \left\| f-f_n \right\| _{L^t}^{t}+\int_K{\left| \left( f_n-g \right) \left( \boldsymbol{x} \right) \right|^td\sigma \left( \boldsymbol{x} \right)}.
		\end{align*}
		Based on our assumption for $\left\{ f_n \right\}$, we know that $\left\| f-f_n \right\| _{L^t}^{t}$ converges to zero when $n\rightarrow \infty $. Hence, we have 
		\begin{align*}
			\int_K{\left| \left( f-g \right) \left( \boldsymbol{x} \right) \right|^td\sigma \left( \boldsymbol{x} \right)}\longrightarrow 0,\quad \left( when\,\,n\rightarrow \infty \right) .
		\end{align*}
		This shows that $f=g\in \mathcal{A} ^t\left( \Omega _D \right)$, which completes the proof.  
	\end{enumerate}
\end{proof}
\begin{theorem}[Plemelj integral formula]\label{Plemelj integral formula}
	Let $\Omega _D\subset \mathbb{R} _{*}^{p+q+1}$ be defined as above with smooth boundary $\partial \Omega _D$, and $x\left( t \right) \in \Omega _D$ is a smooth path in $\mathbb{R} ^{p+q+1}$ and it has non-tangential limit $\boldsymbol{x}\in \partial \Omega _D$ as $t\rightarrow 0$. Then, for each $H\ddot{o}lder$ continuous slice function $f\,:\, \Omega _D\longrightarrow \mathbb{R} _{p+q}$ define on $\Omega _D$, we have
	\begin{align*}
		&\underset{t\rightarrow 0}{\lim}\int_{\partial \Omega _D}{K_{\boldsymbol{y}}\left( x\left( t \right) \right) n\left( \boldsymbol{y} \right) f\left( \boldsymbol{y} \right) d\sigma \left( \boldsymbol{y} \right)}\\
		=&
		\begin{cases}
			\frac{f\left( \boldsymbol{x} \right)}{2}+p.v.\int_{\partial \Omega _D}{K_{\boldsymbol{y}}\left( \boldsymbol{x} \right) n\left( \boldsymbol{y} \right) f\left( \boldsymbol{y} \right) d\sigma \left( \boldsymbol{y} \right)},&x\left( t \right) \in \Omega _D;\\
			-\frac{f\left( \boldsymbol{x} \right)}{2}+p.v.\int_{\partial \Omega _D}{K_{\boldsymbol{y}}\left( \boldsymbol{x} \right) n\left( \boldsymbol{y} \right) f\left( \boldsymbol{y} \right) d\sigma \left( \boldsymbol{y} \right)},&x\left( t \right) \in \mathbb{R} _{*}^{p+q+1}\backslash \Omega _D,
		\end{cases}
	\end{align*}
	where $p.v.$ stands for the principal value.
\end{theorem}
\begin{proof}
	Here, we only present the details for the case $x\left( t \right) \in \Omega _D$, for the other case, one can prove it similarly.\par
	Firstly, it's easy for us to know that 
	\begin{align}\label{Firstly}
		&\int_{\partial \Omega _D}{K_{\boldsymbol{y}}\left( \boldsymbol{x} \right) n\left( \boldsymbol{y} \right) f\left( \boldsymbol{y} \right) d\sigma \left( \boldsymbol{y} \right)}\nonumber\\
		=&\frac{1}{\sigma _{q-1}}\int_{\mathbb{S} ^+}{\int_{\partial \Omega _{\underline{\omega}}}{\mathcal{E} _{\boldsymbol{y}}\left( \boldsymbol{x} \right) n\left( \boldsymbol{y} \right) f\left( \boldsymbol{y} \right) d\sigma_{\underline{\omega}} \left( \boldsymbol{y}\right)}}dS_{\underline{\omega}} \left( \boldsymbol{y}\right) ,
	\end{align}
	where $\underline{\omega}\in \mathbb{S}$. We denote $\boldsymbol{x}\left( t \right) =\boldsymbol{u}\left( t \right) +\underline{\eta }v\left( t \right) $, where $\underline{\eta }\in \mathbb{S}$, $\boldsymbol{u}\left( t \right) \in \mathbb{R} ^{p+1}$, $v\left( t \right) \in \mathbb{R}$ and $\boldsymbol{x}_{\pm \underline{\omega }}\left( t \right) =\boldsymbol{u}\left( t \right) \pm \underline{\omega }v\left( t \right) $. Then, the representation formula givens us that
	\begin{align}\label{Cauchy kernel representation formula}
		&\mathcal{E} _{\boldsymbol{y}}\left( \boldsymbol{x}\left( t \right) \right) \nonumber\\
		&=\frac{1}{2}\left[ \left( \mathcal{E} _{\boldsymbol{y}}\left( \boldsymbol{x}_{\underline{\omega }}\left( t \right) \right) +\mathcal{E} _{\boldsymbol{y}}\left( \boldsymbol{x}_{-\underline{\omega }}\left( t \right) \right) \right) +\underline{\omega }\underline{\eta }\left( \mathcal{E} _{\boldsymbol{y}}\left( \boldsymbol{x}_{-\underline{\omega }}\left( t \right) \right) -\mathcal{E} _{\boldsymbol{y}}\left( \boldsymbol{x}_{\underline{\omega }}\left( t \right) \right) \right) \right] .
	\end{align}
	Hence, we get
	\begin{align}\label{integtal representation formula}
		&\int\limits_{\partial \Omega _{\underline{\omega }}}{\mathcal{E} _{\boldsymbol{y}}\left( \boldsymbol{x}\left( t \right) \right) n\left( \boldsymbol{y} \right) f\left( \boldsymbol{y} \right) d\sigma_{\underline{\omega}} \left( \boldsymbol{y}\right)}\nonumber\\
		=&\frac{1}{2}\bigg[ \int\limits_{\partial \Omega _{\underline{\omega }}}{\mathcal{E} _{\boldsymbol{y}}\left( \boldsymbol{x}_{\underline{\omega }}\left( t \right) \right) n\left( \boldsymbol{y} \right) f\left( \boldsymbol{y} \right) d\sigma_{\underline{\omega}} \left( \boldsymbol{y}\right)}\nonumber\\
		&+\int\limits_{\partial \Omega _{\underline{\omega }}}{\mathcal{E} _{\boldsymbol{y}}\left( \boldsymbol{x}_{-\underline{\omega }}\left( t \right) \right) n\left( \boldsymbol{y} \right) f\left( \boldsymbol{y} \right) d\sigma_{\underline{\omega}} \left( \boldsymbol{y}\right)}\bigg]\nonumber\\	
	   &+\frac{\underline{\omega}\underline{\eta}}{2}\bigg[ \int\limits_{\partial \Omega _{\underline{\omega }}}{\mathcal{E} _{\boldsymbol{y}}\left( \boldsymbol{x}_{-\underline{\omega }}\left( t \right) \right) n\left( \boldsymbol{y} \right) f\left( \boldsymbol{y} \right) d\sigma_{\underline{\omega}} \left( \boldsymbol{y}\right)}\nonumber\\
		&-\int\limits_{\partial \Omega _{\underline{\omega }}}{\mathcal{E} _{\boldsymbol{y}}\left( \boldsymbol{x}_{\underline{\omega }}\left( t \right) \right) n\left( \boldsymbol{y} \right) f\left( \boldsymbol{y} \right) d\sigma_{\underline{\omega}} \left( \boldsymbol{y}\right)} \bigg] .
	\end{align}
	Notice that $\boldsymbol{x}_{\pm \underline{\omega }}\left( t \right)$ approach $\boldsymbol{x}_{\pm \underline{\omega }}$ in $\Omega \cap \mathbb{R} ^{p+2}$ non-tangentially, when $\boldsymbol{x}(t)$ approaches $\boldsymbol{x}$ in $\Omega$ non-tangentially. Hence, with the Sokhotski-Plemelj formula, we have 
	\begin{align*}
		& \underset{t\rightarrow 0}{\lim}\int_{\partial \Omega _{\underline{\omega }}}{\mathcal{E} _{\boldsymbol{y}}\left( \boldsymbol{x}_{\underline{\omega }}\left( t \right) \right) n\left( \boldsymbol{y} \right) f\left( \boldsymbol{y} \right) d\sigma_{\underline{\omega}} \left( \boldsymbol{y}\right)}\\
		=&p.v.\int_{\partial \Omega _D}{K_{\boldsymbol{y}}\left( \boldsymbol{x}_{\underline{\omega}} \right) n\left( \boldsymbol{y} \right) f\left( \boldsymbol{y} \right) d\sigma \left( \boldsymbol{y} \right)}+\frac{f\left( \boldsymbol{x}_{\underline{\omega}} \right)}{2},\\
		& \underset{t\rightarrow 0}{\lim}\int_{\partial \Omega _{\underline{\omega }}}{\mathcal{E} _{\boldsymbol{y}}\left( \boldsymbol{x}_{-\underline{\omega }}\left( t \right) \right) n\left( \boldsymbol{y} \right) f\left( \boldsymbol{y} \right) d\sigma_{\underline{\omega}} \left( \boldsymbol{y}\right)}\\
		=&p.v.\int_{\partial \Omega _D}{K_{\boldsymbol{y}}\left( \boldsymbol{x}_{-\underline{\omega }} \right) n\left( \boldsymbol{y} \right) f\left( \boldsymbol{y} \right) d\sigma \left( \boldsymbol{y} \right)}+\frac{f\left( \boldsymbol{x}_{-\underline{\omega }} \right)}{2},
	\end{align*}
	Plugging the equations above into $(\ref{integtal representation formula})$, we get
	\begin{align*}
		&\int_{\partial \Omega _{\underline{\omega }}}{\mathcal{E} _{\boldsymbol{y}}\left( \boldsymbol{x}\left( t \right) \right) n\left( \boldsymbol{y} \right) f\left( \boldsymbol{y} \right) d\sigma_{\underline{\omega}} \left( \boldsymbol{y}\right)}\\
		=&p.v.\int_{\partial \Omega _D}{K_{\boldsymbol{y}}\left( \boldsymbol{x}_{\underline{\omega }} \right) n\left( \boldsymbol{y} \right) f\left( \boldsymbol{y} \right) d\sigma \left( \boldsymbol{y} \right)}\\
		&+\frac{1}{2}\left[ \left( \frac{f\left( \boldsymbol{x}_{\underline{\omega }} \right)}{2}+\frac{f\left( \boldsymbol{x}_{-\underline{\omega }} \right)}{2} \right) +\underline{\omega }\underline{\eta }\left( \frac{f\left( \boldsymbol{x}_{-\underline{\omega }} \right)}{2}-\frac{f\left( \boldsymbol{x}_{\underline{\omega }} \right)}{2} \right) \right] \\
		=&p.v.\int_{\partial \Omega _D}{K_{\boldsymbol{y}}\left( \boldsymbol{x} \right) n\left( \boldsymbol{y} \right) f\left( \boldsymbol{y} \right) d\sigma \left( \boldsymbol{y} \right)}+\frac{1}{2}f\left( \boldsymbol{x} \right) ,
	\end{align*}
	where the two equations above rely on $(\ref{Cauchy kernel representation formula})$ and the fact that
	\begin{align*}
		\frac{1}{2}\left[ \left( f\left( \boldsymbol{x}_{\underline{\omega }} \right) +f\left( \boldsymbol{x}_{-\underline{\omega }} \right) \right) +\underline{\omega }\underline{\eta }\left( f\left( \boldsymbol{x}_{-\underline{\omega }} \right) -f\left( \boldsymbol{x}_{\underline{\omega }} \right) \right) \right] =f\left( \boldsymbol{x} \right) ,
	\end{align*}
	which comes from Theorem $\ref{*Representation Formula*}$. Therefore, with $(\ref{Firstly})$, we finally have 
	\begin{align*}
		&\underset{t\rightarrow 0}{\lim}\int_{\partial \Omega _D}{K_{\boldsymbol{y}}\left( \boldsymbol{x}\left( t \right) \right) n\left( \boldsymbol{y} \right) f\left( \boldsymbol{y} \right) d\sigma \left( \boldsymbol{y} \right)}\\
		=&\frac{f\left( \boldsymbol{x} \right)}{2}+p.v.\int_{\partial \Omega _D}{K_{\boldsymbol{y}}\left( \boldsymbol{x} \right) n\left( \boldsymbol{y} \right) f\left( \boldsymbol{y} \right) d\sigma \left( \boldsymbol{y} \right)},
	\end{align*}
	which completes the proof.
\end{proof}
\begin{corollary}\label{corollary of  Plemelj formula}
	Let $\Omega _D\subset \mathbb{R} _{*}^{p+q+1}$ be a bounded p-symmetric domain with smooth boundary $\partial \Omega _D$. The relation 
	\begin{align*}
		p.v.\int_{\partial \Omega _D}{K_{\boldsymbol{y}}\left( \boldsymbol{x} \right) n\left( \boldsymbol{y} \right) g\left( \boldsymbol{y} \right) d\sigma \left( \boldsymbol{y} \right)}=\frac{g\left( \boldsymbol{x} \right)}{2},\quad for\,\,all\,\,x\in \partial \Omega _D
	\end{align*}
	is necessary and sufficient so that g represents the boundary values of a generalized partial-slice monogenic function \\
	\begin{align*}
		p.v.\int_{\partial \Omega _D}{K_{\boldsymbol{y}}\left( \boldsymbol{x} \right) n\left( \boldsymbol{y} \right) g\left( \boldsymbol{y} \right) d\sigma \left( \boldsymbol{y} \right)}
	\end{align*}
	defined in $\Omega _D$.
	On the other hand, the relation 
	\begin{align*}
		p.v.\int_{\partial \Omega _D}{K_{\boldsymbol{y}}\left( \boldsymbol{x} \right) n\left( \boldsymbol{y} \right) g\left( \boldsymbol{y} \right) d\sigma \left( \boldsymbol{y} \right)}=-\frac{g\left( \boldsymbol{x} \right)}{2},\quad for\,\,all\,\,x\in \partial \Omega _D
	\end{align*}
	is necessary and sufficient so that g represents the boundary values of a generalized partial-slice monogenic function \\
	\begin{align*}
		p.v.\int_{\partial \Omega _D}{K_{\boldsymbol{y}}\left( \boldsymbol{x} \right) n\left( \boldsymbol{y} \right) g\left( \boldsymbol{y} \right) d\sigma \left( \boldsymbol{y} \right)}
	\end{align*}
	defined in $\mathbb{R} _{*}^{p+q+1}\backslash \Omega _D$.
\end{corollary}
\begin{proof}
	Let $f$ be the generalized partial-slice monogenic continuation into the domain $\Omega _D$ of the function $g$ given on $\partial \Omega _D$. From the Cauchy integral formula (\ref{CIF}), we know that
	\begin{align*}
		f\left( \boldsymbol{x} \right) =\int_{\partial \Omega _D}{K_{\boldsymbol{y}}\left( \boldsymbol{x} \right) n\left( \boldsymbol{y} \right) g\left( \boldsymbol{y} \right) d\sigma \left( \boldsymbol{y} \right)}.
	\end{align*}
	Therefore, the non-tangential boundary values of $f$ are $g$. Applying the Plemelj formula introduced in Theorem \ref{Plemelj integral formula}, we have
	\begin{align*}
		g\left( \boldsymbol{x} \right) =\frac{g\left( \boldsymbol{x} \right)}{2}+p.v.\int_{\partial \Omega _D}{K_{\boldsymbol{y}}\left( \boldsymbol{x} \right) n\left( \boldsymbol{y} \right) g\left( \boldsymbol{y} \right) d\sigma \left( \boldsymbol{y} \right)},\quad for\,\,all\,\,x\in \partial \Omega _D,
	\end{align*}
	which leads to 
	\begin{align*}
		p.v.\int_{\partial \Omega _D}{K_{\boldsymbol{y}}\left( \boldsymbol{x} \right) n\left( \boldsymbol{y} \right) g\left( \boldsymbol{y} \right) d\sigma \left( \boldsymbol{y} \right)}=\frac{g\left( \boldsymbol{x} \right)}{2},\quad for\,\,all\,\,x\in \partial \Omega _D.
	\end{align*}
	If vice versa, we have
	\begin{align*}
		p.v.\int_{\partial \Omega _D}{K_{\boldsymbol{y}}\left( \boldsymbol{x} \right) n\left( \boldsymbol{y} \right) g\left( \boldsymbol{y} \right) d\sigma \left( \boldsymbol{y} \right)}=\frac{g\left( \boldsymbol{x} \right)}{2},\quad for\,\,all\,\,x\in \partial \Omega _D.
	\end{align*}
	The Plemelj formula implies that 
	\begin{align*}
		\int_{\partial \Omega _D}{K_{\boldsymbol{y}}\left( \boldsymbol{x} \right) n\left( \boldsymbol{y} \right) g\left( \boldsymbol{y} \right) d\sigma \left( \boldsymbol{y} \right)}
	\end{align*}
	has the boundary value $g$. Therefore, it is the generalized partial-slice monogenic continuation of $g$ into $\Omega _D$. The proof for the exterior domain case can be obtained similarly. 
\end{proof}
Now, we introduce an integral operator as follows.
\begin{align*}
	\left( S_{\partial \Omega}u \right) \left( \boldsymbol{x} \right) :=2\int_{\partial \Omega _D}{K _{\boldsymbol{y}}\left( \boldsymbol{x} \right) n\left( \boldsymbol{y} \right) u\left( \boldsymbol{y} \right) d\sigma \left( \boldsymbol{y} \right)},\quad \boldsymbol{x}\in \Omega _D.
\end{align*}
Next, with the Plemelj-Sokhotski formula given above, one can easily obtain a result on generalized partial-slice monogenic continuation as follows.
\begin{corollary}\label{Cor}
	Given that $u$ is H\"older continuous on $\partial \Omega _D$, we obtain the algebraic identity $S_{\partial \Omega _D}^{2}u = Iu$, where $I$ denotes the identity operator.
\end{corollary}
\par
Define the Plemelj projections $P_{\partial \Omega _D}$ and $Q_{\partial \Omega _D}$ as  
\begin{align*}
	P_{\partial \Omega _D} = \frac{1}{2}(I + S_{\partial \Omega _D}), \quad Q_{\partial \Omega _D} = \frac{1}{2}(I - S_{\partial \Omega _D}).  
\end{align*}
We will observe that $P_{\partial \Omega _D}$ projects onto the space of all defined functions. Then according to the conclusion of corollary \ref{Cor}, we can draw the following corollary.
\begin{corollary}
	The operators $P_{\partial \Omega _D}$ and $Q_{\partial \Omega _D}$ project onto spaces of functions. Specifically, $P_{\partial \Omega _D}$ projects onto functions defined on $\partial \Omega _D$ that are holomorphically continuable into $\Omega _D^{+}$. Conversely, $Q_{\partial \Omega _D}$ projects onto functions holomorphically continuable into $\Omega _D^{-}$ that vanish at $\infty$. These operators satisfy the following algebraic properties
\begin{align*}
	P_{\partial \Omega _D}^{2}=P_{\partial \Omega _D}, \quad Q_{\partial \Omega _D}^{2}=Q_{\partial \Omega _D},\quad P_{\partial \Omega _D}Q_{\partial \Omega _D}=Q_{\partial \Omega _D}P_{\partial \Omega _D}=0.
\end{align*}
\end{corollary}
\begin{proof}
	This conclusion is derived directly from the definition and is further substantiated by the Plemelj-Sokhotski formula along with its associated implications.
\end{proof}
\begin{theorem}[Hodge decomposition]
	Let $\Omega _D\subset \mathbb{R} _{*}^{p+q+1}$ be a bounded $p$-symmetric domain and $t>q$. Then, the space $\mathcal{L} ^t\left( \Omega _D \right)$ allows the orthogonal decomposition 
	\begin{align*}
		\mathcal{L} ^t\left( \Omega _D \right) =\mathcal{A} ^t\left( \Omega _D \right) \oplus \left( |\underline{\boldsymbol{y}}_q|^{1-q}\bar{\vartheta}\mathcal{L}_{0}^{t}\left( \Omega _D \right)  \right) 
	\end{align*}
	with respect to the $\mathbb{R} _{p+q}$-valued inner product given by
	\begin{align*}
		\langle f,g\rangle :=\int_{\Omega _D}{\overline{f\left( \boldsymbol{y} \right) }g\left( \boldsymbol{y} \right) d\sigma \left( \boldsymbol{y} \right)},\quad for\,\,all\,\,f,g\in L ^t\left( \Omega _D \right).
	\end{align*}
\end{theorem}
\begin{proof}
	Define $X=\mathcal{L} ^t\left( \Omega _D \right) \ominus \mathcal{A} ^t\left( \Omega _D \right)$ as the orthogonal complement of the space $\mathcal{A} ^t\left( \Omega _D \right)$, with respect to the inner product $\langle\cdot ,\cdot \rangle$ that has been specified previously. Given any function $f\in X$, we have $|\underline{\boldsymbol{y}}_q|^{q-1}f\in L^t\left( \Omega _D \right) $, so that $g=T_{\Omega _D}\left( |\underline{\boldsymbol{y}}_q|^{q-1}f \right) \in L^{t}\left( \Omega _D \right) $ as well. Subsequently, we have $f\left( \boldsymbol{y} \right) =\frac{\bar{\vartheta}g\left( \boldsymbol{y} \right)}{|\underline{\boldsymbol{y}}_q|^{q-1}}$, and for any $\phi \in \mathcal{A} ^t\left( \Omega _D \right)$, we have
	\begin{align*}
		0=\int_{\Omega _D}{\overline{\phi \left( \boldsymbol{y} \right) }f\left( \boldsymbol{y} \right) d\sigma \left( \boldsymbol{y} \right)}=\int_{\Omega _D}{\overline{\phi \left( \boldsymbol{y} \right) }\frac{\bar{\vartheta}g\left( \boldsymbol{y} \right)}{|\underline{\boldsymbol{y}}_q|^{q-1}}d\sigma \left( \boldsymbol{y} \right)}.
	\end{align*}
	Specifically, consider the function $\phi _l\left( \boldsymbol{y} \right) =\frac{1}{\sigma _{q-1}}\overline{\mathcal{E} _{\boldsymbol{y}}\left( \boldsymbol{y}_l \right) }$, where the set of $\left\{ \boldsymbol{x}_l \right\}$ is dense in $\mathbb{R} _{*}^{p+q+1}\backslash \overline{\Omega _D}$. Clearly, we have $\overline{\phi _l\left( \boldsymbol{y} \right) }|\underline{\boldsymbol{y}}_q|^{1-q}=K_{\boldsymbol{y}}\left( \boldsymbol{x}_l \right) $, $\phi _l\left( \boldsymbol{y} \right) \in \mathcal{L} ^t\left( \Omega _D \right) $ and $\phi _l\left( \boldsymbol{y} \right) \overline{\vartheta }_{\boldsymbol{y}}=0$, where $\overline{\vartheta }_{\boldsymbol{y}}$ means that $\overline{\vartheta}$ is a differential operator with respect to $\boldsymbol{y}$. Subsequently, employing Gauss's theorem, we deduce that
	\begin{align*}
		0=&\int_{\Omega _D}{\overline{\phi _l\left( \boldsymbol{y} \right) }\frac{\bar{\vartheta}g\left( \boldsymbol{y} \right)}{|\underline{\boldsymbol{y}}_q|^{q-1}}d\sigma \left( \boldsymbol{y} \right)}\\
		=&\frac{1}{\sigma _{q-1}}\int_{\mathbb{S} ^+}{\int_{\Omega _{\underline{\omega }}}{\mathcal{E} _{\boldsymbol{y}}\left( \boldsymbol{x}_l \right) \frac{\bar{\vartheta}g\left( \boldsymbol{y} \right)}{|\underline{\boldsymbol{y}}_q|^{q-1}}d\sigma _{\underline{\omega }}\left( \boldsymbol{y} \right)}dS_{\underline{\omega}} \left( \boldsymbol{y}\right)}\\
		=&-\frac{1}{\sigma _{q-1}}\int_{\mathbb{S} ^+}{\bigg[ \int_{\Omega _{\underline{\omega}}}{\left( \mathcal{E} _{\boldsymbol{y}}\left( \boldsymbol{x}_l \right) \bar{\vartheta} \right) g\left( \boldsymbol{y} \right) |\underline{\boldsymbol{y}}_q|^{1-q}d\sigma _{\underline{\omega}}\left( \boldsymbol{y} \right)}}\\
		&-\int_{\partial \Omega _{\underline{\omega}}}{\mathcal{E} _{\boldsymbol{y}}\left( \boldsymbol{x}_l \right) n\left( \boldsymbol{y} \right) g\left( \boldsymbol{y} \right) |\underline{\boldsymbol{y}}_q|^{1-q}d\sigma _{\underline{\omega}}\left( \boldsymbol{y} \right)} \bigg]dS_{\underline{\omega}} \left( \boldsymbol{y}\right) \\
		=&\frac{1}{\sigma _{q-1}}\int_{\mathbb{S} ^+}{\int_{\partial \Omega _{\underline{\omega}}}{\mathcal{E} _{\boldsymbol{y}}\left( \boldsymbol{x}_l \right) n\left( \boldsymbol{y} \right) g\left( \boldsymbol{y} \right) |\underline{\boldsymbol{y}}_q|^{1-q}d\sigma _{\underline{\omega}}\left( \boldsymbol{y} \right)}dS_{\underline{\omega}} \left( \boldsymbol{y}\right)}\\
		=&F_{\partial \Omega _D}\left( trg \right) \left( \boldsymbol{x}_l \right),
	\end{align*}
	where $trg$ denotes the trace of $g$. Hence, due to continuity considerations, it follows that $F_{\partial \Omega _D}\left( trg \right) =0$ in $\mathbb{R} _{*}^{p+q+1}\backslash \overline{\Omega _D}$. Then, the Plemelj formula as stated in Theorem \ref{Plemelj integral formula} informs us that 
	\begin{align*}
		p.v.\int_{\partial \Omega _D}{K_{\boldsymbol{y}}\left( \boldsymbol{x} \right) n\left( \boldsymbol{y} \right) trg\left( \boldsymbol{y} \right) d\sigma \left( \boldsymbol{y} \right)}=\frac{trg\left( \boldsymbol{y} \right)}{2},\quad for\,\,all\,\,\boldsymbol{y}\in \partial \Omega _D.
	\end{align*}
	Hence, with Corollary \ref{corollary of  Plemelj formula}, the trace $trg$ can be  generalized partial-slice monogenicly extended into the domain $\Omega _D$. Here, we use $h$ to denote the continuation. Then, we have $tr_{\partial \Omega _D}g=tr_{\partial \Omega _D}h$ and the trace operator $tr_{\partial \Omega _D}$ describes the restriction onto the boundary $\partial \Omega _D$.\par
	Next, let $\omega :=g-h$, then we have $tr_{\partial \Omega _D}\omega =0$, in other words, $\omega \in W_{0}^{1,t}\left( \Omega _D \right)$. Further, we can see that
	\begin{align*}
		\bar{\vartheta}\omega =\bar{\vartheta}g=\bar{\vartheta}T_{\Omega _D}|\underline{\boldsymbol{y}}_q|^{q-1}f=|\underline{\boldsymbol{y}}_q|^{q-1}f\left( \boldsymbol{x} \right) .
	\end{align*}
	Indeed, since $f\in \mathcal{GS} \left( \Omega _D \right)$, we suppose $f$ is induced by the stem function $F=F_1+iF_2$, where $F_1$, $F_2$ satisfy the even-odd condition given in $(\ref{Stem Function Costituent})$. One can easily check that for $\boldsymbol{y}=\boldsymbol{y}_p+r\underline{\omega}$, $\boldsymbol{y}_{\diamond}=\boldsymbol{y}_p-r\underline{\omega }$ the functions $H_1$, $H_2$ defined by
	\begin{align*}
		H_1:=\left| \frac{\boldsymbol{y}-\boldsymbol{y}_{\diamond}}{2\underline{\omega}} \right|^{q-1}F_1,\quad H_2:=\left| \frac{\boldsymbol{y}-\boldsymbol{y}_{\diamond}}{2\underline{\omega}} \right|^{q-1}F_2,
	\end{align*}
	also satisfy the even-odd conditions. Further, $|\underline{\boldsymbol{y}}_q|^{q-1}f$ is induced by the stem function $H=H_1+iH_2$, which justifies that $\bar{\vartheta}\omega =\bar{\vartheta}g=|\underline{\boldsymbol{y}}_q|^{q-1}f\left( \boldsymbol{y} \right)\in \mathcal{GS} \left( \Omega _D \right)$ and this completes the proof.
\end{proof}   
\subsection*{Acknowledgments}
The work of Chao Ding is supported by National Natural Science Foundation of China (No. 12271001), Natural Science Foundation of Anhui Province (No. 2308085MA03) and Excellent University Research and Innovation Team in Anhui Province (No. 2024AH010002).


\subsection*{Data Availability}
No new data were created or analysed during this study. Data sharing is not applicable to this article.



\begin{thebibliography}{1}
	\bibitem{Br1} \textsc{F. Brack, R. Delanghe, F. Sommen}, \textit{Clifford Analysis}, (Research Notes in Mathematics 76), Pitman Books Ltd, London, 1982.
	
	\bibitem{Br2} \textsc{F. Brackx, R. Delanghe, F. Sommen}, \textit{Clifford analysis}, Research Notes in Mathematics, Vol. 76, Pitman, Boston, 1982.
	
	\bibitem{Co1} \textsc{F. Colombo, F. Sommen}, Distributions and the global operator of slice monogenic functions, \textit{Complex Anal. Oper. Theory} \textbf{8} (2014), no. 6, 1257-1268.
	
	\bibitem{Co2} \textsc{F. Colombo, I. Sabadini, D. C. Struppa}, Slice monogenic functions, \textit{Israel J.Math.} \textbf{171}(2009), 385-403.
	
	\bibitem{Co3} \textsc{F. Colombo, I. Sabadini, D. C. Struppa}, Entire Slice Regular Functions, \textit{SpringerBriefs in Mathematics}, Springer, Cham, (2016).
	
	\bibitem{Co4}\textsc{F. Colombo, I. Sabadini, D.C. Struppa}, \textit{Noncommutative Functional Calculus, Theory and Applications of Slice Hyperholomorphic Functions}, Progress in Mathematics 289, Birkh\"auser , 2011.
	
	\bibitem{Co5} \textsc{F. Colombo, I. Sabadini, D.C. Struppa}, An extension theorem for slice monogenic functions and some of its consequences, \textit{Israel J. Math.} \textbf{177}(2010), 369-489.
	
	\bibitem{Co6} \textsc{F. Colombo, J. O. Gonz\'alez-Cervantes, I. Sabadini}, \textit{A nonconstant coefficients differential operator associated to slice monogenic functions}, Trans. Amer. Math. Soc. 365 (2013), no. 1, 303-318.
	
	\bibitem{Co7} F. Colombo, J.O. Gonz\'alez-Cervantes, M.E. Luna-Elizarrar\' as, I.Sabadini, M. Shapiro, \emph{On Two Approaches to the Bergman Theory for Slice Regular Functions}, in \emph{Advances in Hypercomplex Analysis}, Springer INdAM Series, vol 1. Springer, Milano, 2013. 
	
	\bibitem{Co8} \textsc{F. Colombo, I. Sabadini, D. C. Struppa}, Michele Sce’s Works in Hypercomplex Analysis. A Translation with Commentaries, \textit{Birkh\"auser}, Basel (2020).
	
	\bibitem{De} R. Delanghe, F. Sommen, V. Sou\v cek: \emph{Clifford Analysis and Spinor Valued Functions}, Kluwer Academic Dordrecht, 1992.
	
	\bibitem{Di1} C. Ding, X.Q. Cheng: \emph{Integral formulas for slice Cauchy-Riemann operator and applications}, Adv. Appl. Clifford Algebras 34, 32 (2024).
	
	\bibitem{Di2} \textsc{P. A. M. Dirac}, The quantum theory of electron, \textit{Proc. Roy. Soc} \textbf{A117}, pp. 610-624, 1928.
	
	\bibitem{Fo} G. Folland: \emph{Real Analysis: Modern Techniques and Their Applications}, 2nd edition, Wiley, 2007.
	
	\bibitem{Fu} \textsc{R. Fueter}, Die Funktionentheorie der Differentialgleichungen $\Delta u=0$ und $\Delta\Delta u=0$ mit vier reellen Variablen, (German), \textit{Comment. Math. Helv.} \textbf{7} (1934), no. 1, 307-330.
	
	\bibitem{Gu} \textsc{K. G\"{u}rlebeck, K. Habetha, and W. Spr\"{o}ßig}, \textit{Holomorphic functions in the plane and n-dimensional space}, Birkh\"{a}user Verlag, Basel, 2008.
	
	\bibitem{Ge1} \textsc{G. Gentili and D. C. Struppa}, A new approach to Cullen-regular functions of a quaternionic variable, \textit{Comptes Rendus Math\'ematique Acad\'emie des Sciences}, \textbf{342}(2006), 741-744.
	
	\bibitem{Ge2} \textsc{G. Gentili, C. Stoppato, D. C. Struppa, F. Vlacci}, Recent developments for regular functions of a hypercomplex variable, in: \textit{Hypercomplex Analysis}, in: \textit{Trends Math.}, Birkh\"auser, Basel, 2009, pp. 165-186.
	
	\bibitem{Ge3} \textsc{G. Gentili, D. C. Struppa}, Regular functions on the space of Cayley numbers, \textit{Rocky Mt. J. Math.} \textbf{40} (2010), 225-241.
	
	\bibitem{Gh1}\textsc{R. Ghilnoi, A. Perotti}, Slice regular functions on real alternative algebras, \textit{Adv. Math.}, \textbf{226}(2011), 1662-1691.
	
	\bibitem{Gh2} \textsc{R. Ghiloni, A. Perotti}, Global differential equations for slice regular functions, \textit{Math. Nachr.} \textbf{287}(2014), 561-573.
	
	\bibitem{Hu} M. Hu, C. Ding, Y. Shen, J. Wang, \emph{Integral formulas and Teodorescu transform for generalized partial-slice monogenic functions}, arXiv:2502.20737, 2025.
	
	\bibitem{Ji} \textsc{M. Jin, G. Ren, I. Sabadini}, Slice Dirac operator over octonions, (2019), to appear in \textit{Israel J. Math.}, arXiv:1908.01383.
	
	\bibitem{Pe} A. Perotti: \emph{Almansi Theorem and Mean Value Formula for Quaternionic Slice regular Functions}, Adv. Appl. Clifford Algebras, Vol. 30, article number 61,2020.
	
	\bibitem{Ry} \textsc{J. Ryan}, Conformal Clifford manifolds arising in Clifford analysis, \textit{Proc. Roy. Irish Acad. Sect. A} \textbf{85} (1985), no. 1, 1-23.
	
	\bibitem{Ve} \textsc{I. N. Vekua}, \textit{Generalized Analytic Functions}, Oxford: Pergamon Press, 1962.
	
	\bibitem{Xu1}\textsc{Z. Xu,  I. Sabadini}, Generalized Partial-Slice Monogenic Functions: A Synthesis of Two Function Theories, \textit{Adv. Appl. Clifford Algebras} \textbf{34}(2024), article number 10.
	
	\bibitem{Xu2} \textsc{Z. Xu, I.Sabadini}, Generalized partial-slice monogenic functions, \textit{Trans. Amer. Math. Soc.} \textbf{378} (2025), pp. 851-883.
	\bibitem{Zh1} \textsc{Z. Zhang}, Integral Representations and its Applications in Clifford Analysis, \textit{General Mathematics} \textbf{13}, No. 3, pp. 81-98, 2005.
	
	\bibitem{Zh2} \textsc{Z. Zhang}, On k-regular functions with values in a universal Clifford algebra, \textit{Journal of Mathematical Analysis and Applications} \textbf{315}, pp. 491-505, 2006.
	
\end{thebibliography}

\end{document}